\documentclass{article}

\usepackage[english]{babel}
\usepackage[letterpaper,top=2cm,bottom=2cm,left=2cm,right=2cm,marginparwidth=1.75cm]{geometry}

\usepackage{authblk}
\usepackage{amsmath}
\usepackage{amssymb}
\usepackage{graphicx}
\usepackage{xcolor}
\usepackage{caption}
\usepackage{subcaption}
\usepackage{bm}
\usepackage{float}
\usepackage[colorlinks=true, allcolors=blue]{hyperref}
\usepackage{lipsum}
\usepackage{multicol}
\usepackage{multirow}
\usepackage{ulem}
\usepackage[numbers]{natbib}
\usepackage{amsthm}
\usepackage{threeparttable}

\newcommand{\ud}[1]{\,\mathrm{d}#1}
\providecommand{\keywords}[1]{\small\\\textbf{\textit{Keywords:}} #1}

\newtheorem{theorem}{Theorem}
\newtheorem{definition}{Definition}
\newtheorem{lemma}{Lemma}
\newtheorem{remark}{Remark}

\definecolor{blue}{rgb}{0, 0, 0}

\usepackage{setspace}
\usepackage[modulo]{lineno}

\title{Topology optimization including a model of the {\color{blue}layer-by-layer} additive manufacturing process}

\author[1]{G.A.~Haveroth}       %Haveroth, Geovane Augusto
\author[2]{C.-J.~Thore}         %Thore, Carl-Johan;
\author[4]{M.R.~Correa}         %Correa, Maicon Ribeiro;
\author[1]{R.F.~Ausas}          %Ausas, Roberto Federico;
\author[3]{S.~Jakobsson}        %Jakobsson, Stefan; 
\author[1]{J.A.~Cuminato}       %Cuminato, José Alberto;
\author[2]{A.~Klarbring}        %Klarbring, Anders;
\affil[1]{Department of Applied Mathematics and Statistics, Institute of Mathematics and Computer Sciences, University of São Paulo, São Carlos-SP, Brazil}
\affil[2]{Division of Solid Mechanics, Department of Management and Engineering, Institute of Technology, Linköping University, 581 83 Linköping, Sweden}
\affil[3]{Arcam EBM, Mölndal, Sweden}
\affil[4]{Department of Applied Mathematics, Institute of Mathematics, Statistics and Scientific Computing, University of Campinas, Campinas-SP, Brazil}

\date{\today}

\AtBeginDocument{\let\latexlabel\label}

\begin{document}
\maketitle

\begin{abstract}
A topology optimization formulation including a model of the layer-by-layer additive manufacturing (AM) process is considered. Defined as a multi-objective minimization problem, the formulation accounts for the performance and cost of both the final and partially manufactured designs and allows for considering AM-related issues such as overhang and residual stresses in the optimization. 
The formulation is exemplified by stiffness optimization in which the overhang is limited by adding mechanical or thermal compliance as a measure of the cost of partially manufactured designs. 
Convergence of the model as the approximate layer-by-layer model is refined is shown theoretically, and an extensive numerical study indicates that this convergence can be fast, thus making it a computationally viable approach useful for including AM-related issues into topology optimization. The examples also show that drips and sharp corners associated with some geometry-based formulations for overhang limitation can be avoided. The codes used in this article are written in Python using only open sources libraries and are available for reference.
\noindent \keywords{Topology optimization; Additive manufacturing; Layer-by-Layer process; Thermal conductivity; Self-weight; Stiffness.}
\end{abstract}

\section{Introduction}
\label{sec:introduction}

Mechanical and structural engineers always strive to make efficient use of materials, e.g., by developing lighter structures with the same load-carrying capacity, thus bringing economical and environmental benefits. The search for more efficient structures was originally based on a trial-and-error process. However, in the last decades, computational tools based on optimization theory have been developed to find optimal structures semi-automatically. In particular, Topology Optimization (TO) \cite{bendsoe2003topology} 
has gained a lot of interest for its ability to generate highly efficient designs that are superior to those developed by traditional trial-and-error approaches. However, optimality often comes at the cost of geometric complexity, implying that topology optimized designs may be impossible to manufacture using traditional methods. Therefore, researchers and users of TO have become interested in additive manufacturing (AM) as a way to realize complex TO designs.
{\color{blue} In particular, the powder bed fusion (PBF), the  electron beam melting (EBM) and the selective laser melting (SLM), technologies where a focused energy beam (electron or laser) selectively melts the fine metal powder, have attracted interest.}

While offering great freedom with regards to which designs can be manufactured, all AM processes have limitations and characteristics -- minimal printable sizes, allowable material deposition angle (maximum overhang angle), bridging distance, heat transfer accommodation \cite{leary2014optimal, pellens2019combined, ameen2019self} and so on -- which should be accounted for in the design process to ensure high-quality builds, thus motivating research into TO formulations aimed specifically at AM.

In Much work on TO for AM has focused on avoiding or limiting overhangs. Two main approaches to this problem can be identified: (i) support structures are optimized in a post-processing step for a given design \cite{mezzadri2018topology, cheng2019utilizing}; and (ii) overhang constraints are included in the TO to ensure manufacturability without support \cite{Langelaar:2017, qian2017undercut, allaire2017structural, langelaar2018combined, garaigordobil2019overhang, thore2019penalty, mezzadri2020second, vdVen:2021}. The second approach is the most attractive, ensuring simultaneously optimality and manufacturability. The proposed methods for overhang control in this approach can be divided into geometry- and physics-based. Geometry-based approaches are computationally cheap but may lead to inefficient structures that are not self-supported or exhibits drips and sharp corners which are not optimal and may lead to high stresses \cite{qian2017undercut,thore2019penalty}. In physics-based approaches, one tries to mimic the layer-by-layer build process. This is also the case in some geometry-based approaches \cite{Langelaar:2017}, but physics-based approaches also include an explicit model of the physics
\cite{allaire2017structural,allaire2018optimizing,amir2018topology,wang2020space}. Depending on the complexity of the physics model this can be more computationally costly than geometry-based approaches but may also lead to designs that perform better and are more suitable for manufacturing. We also emphasize that having a physical model of the layer-by-layer process allows for consideration of, for example, path planning or residual stresses
\cite{Allaire:2018,miki2021topology} that develop during the build-processes and may lead to part distortion and worsened strength and fatigue properties.

In this article, we present a general TO formulation that incorporates ideas from the physical AM layer-by-layer building process. Defined as a multi-objective minimization problem, which we refer to as the \textit{archetype problem} below, the formulation accounts for the performance and cost of both the final and partially manufactured designs and allows for considering AM-related issues such as overhang and residual stresses in the optimization. The formulation is exemplified by -- but certainly not limited to! -- stiffness optimization in which the overhang is controlled by adding mechanical or thermal compliance as a measure of the cost of partially manufactured designs.
{\color{blue}
In the proposed model of the building process, we solve separate boundary values problems (BVPs) for each partially built structure and one for the final design.
The BVPs have separate load cases, as detailed in Section \ref{sec:archetype_problem}.}

The remainder of this article is organized as follows: Section \ref{sec:classical_to_formulation} presents the standard density-based TO formulation. Section \ref{sec:archetype_problem} introduces the archetype TO problem together with two specializations based on self-weight and thermal conductivity.
Section \ref{sec:num_results} presents extensive numerical studies in 2D and 3D using the proposed formulations. Finally, conclusions are given in Section \ref{sec:conclusion}.

\section{Standard density-based TO formulation}
\label{sec:classical_to_formulation}

Let the design domain $\Omega \subset \mathbb{R}^d$ ($d = 2,3$) be an open, bounded and connected set with Lipschitz continuous boundary $\Gamma$ divided into two disjoint parts $\Gamma_D$, $|\Gamma_D|>0$, and $\Gamma_N$ where the body is subject to Dirichlet and Neumann conditions, respectively.
{\color{blue}
For simplicity and without loss of generality, homogeneous Dirichlet boundary conditions are considered throughout the paper. In all the text, the vector $\bm{x}$ designates the coordinate of a point in $\Omega$.}
We want to determine the domain of an isotropic elastic material, the so-called design (-solid) domain, $\Omega^S \subset \Omega$, that satisfies desirable features described by appropriate constraints.
The standard formulation of density-based TO for stiffness maximization in continuous variational form reads:
\begin{equation}
    \begin{array}{cll}
        \displaystyle \min_{\substack{ \rho \in L^\infty \left( \Omega, [0,1] \right)\\ \bm{u} \in V }} & 
        a({\rho};\bm{u},\bm{u}) & \textmd{[compliance $=J_D$]} \\
        \textmd{s.t.} 
& a({\rho};\bm{u},\bm{v}) = L(\bm{v}) \quad \forall \bm{v} \in V & \textmd{[balance equation]}\\
& \displaystyle \int_{\Omega} \bar{\rho} \left( \rho \right) \, \mathrm{d}x \leq \bar{v} \int_{\Omega} \, \mathrm{d}x & \textmd{[volume constraint]} \\
    \end{array}
    \label{eq:standad_formulation_TO}
\end{equation}
% with trial and test spaces equally defined as
where
\begin{equation*}
    V = \left\{ \bm{v} \in {\left[ H^1(\Omega) \right]}^d \ \vert \ \bm{v} = \bm{0} \ \textmd{on} \ \Gamma_D \right\} .
\end{equation*}

\noindent
The compliance $J_D$ is the cost function, $\bar{v}$ is the fraction of solid material allowed in $\Omega$, and $\rho$ is the optimization variable or design field prior to the filtering and threshold projection procedures (described below); that is,
\begin{eqnarray*}
{\color{blue}
    \rho (\bm{x}) \overset{\mathrm{filtering}}{\longrightarrow} 
    \hat{\rho} = \hat{\rho}(\rho)(\bm{x})
    \overset{\mathrm{threshold}}{\longrightarrow} 
    \bar{\rho} = \bar{\rho} 
    \left(  
    \hat{\rho}(\bm{x})
    \right), \quad \bm{x} \in \Omega.}
\end{eqnarray*}
The bilinear and linear forms are defined respectively by
\begin{equation}
    a({\rho}; \bm{u}, \bm{v}) = \int_{\Omega} \bm{\sigma} ({\rho}, \bm{u}) : \bm{\varepsilon} (\bm{v}) \, \mathrm{d}x
    \quad
    \textmd{and}
    \quad
    L(\bm{v}) = \int_{\Gamma_N} \bm{t} \cdot \bm{v} \, \mathrm{d}s,
    \label{eq:varastandard}
\end{equation}

\noindent
where $\bm{\varepsilon} (\bm{v}) = \nabla^S \bm{v}$ is the symmetric gradient tensor, $\bm{\sigma} (\rho, \bm{u})$ is the stress tensor {\color{blue}given by}
\begin{eqnarray}
    \bm{\sigma}({\rho}, \bm{u}) 
    &=& \lambda \textmd{tr} \left( \bm{\varepsilon} (\bm{u}) \right) \bm{I} + 2 \mu \bm{\varepsilon} (\bm{u}) \nonumber \\
    &=& \mathbb{D}( \lambda, \mu ) : \bm{\varepsilon} (\bm{u}),
    \label{eq:stress_tensor_expression}
\end{eqnarray}
and $\bm{t}$ is a given traction vector acting on the Neumann boundary $\Gamma_N$.
Herein, $\mathbb{D}$ is the symmetric fourth-order elasticity tensor with components $\mathbb{D}_{ijkl}=\lambda \delta_{ij}\delta_{kl}+\mu\left( \delta_{ik}\delta_{jl}+\delta_{il}\delta_{jk} \right)$.
The Lam\'e elasticity material parameters, $\lambda = \lambda(E({\rho}), \nu)$ and $\mu = \mu(E({\rho}), \nu)$, are interpolated in $\Omega$ according to the solid isotropic material with penalization (SIMP) \cite{bendsoe1989optimal} scheme as
\begin{equation*}
    E(\bar{\rho}\left( \rho \right)) = E_{\min} + \bar{\rho}^q \left( E_0 - E_{\min} \right) ,
    \label{eq:SIMP_interpolation}
\end{equation*}
where $E$ and $\nu$ are the Young's modulus and Poisson ratio, respectively.
{\color{blue}
In addition, $E_{\min}$ and $E_0$ represent the Young's minimum (void) and full-material values, respectively, with $0 < E_{\min} < E_{0}$ and $q$ is an appropriate penalty factor.}
Let us recall that the SIMP method is not well
suited for TO when self-weight are included due to parasitic eﬀects arising at low densities \cite{bruyneel2005note}. An alternative scheme, that avoids such pathology, is the rational approximation of material properties (RAMP) \cite{stolpe2001alternative}, in which the Young's modulus is replaced by
\begin{equation*}
    E(\bar{\rho}\left( \rho \right)) = E_{\min} + \frac{\bar{\rho}}{1+q(1-\bar{\rho})} \left( E_0 - E_{\min} \right) .
    \label{eq:RAMP_interpolation}
\end{equation*}
Hereafter, we use the SIMP strategy unless stated otherwise;
When it is the case, the RAMP scheme is identified by 
{\color{blue} the superscription ${\left( \cdot \right)}^\mathrm{RAMP}$.}

\begin{remark}
The standard density-based TO formulation \eqref{eq:standad_formulation_TO} can be conveniently redefined into a problem of the form known as the nested form in the TO-literature.
It is obtained by using the balance equation to express $\bm{u}$ as a function of $\rho$, leading to the problem
\begin{equation*}
    \displaystyle \min_{ \rho \in \mathcal{D}} L(\bm{u}(\rho)), \\
    \label{eq:standad_formulation_TO_nested_form}
\end{equation*}
where the set of admissible designs $\mathcal{D}$ comprise those designs satisfying the box ($\rho\in[0,1]$ a.e.) and volume constraints in \eqref{eq:standad_formulation_TO}, and $\bm{u}(\rho)$ is the unique solution to the state problem for a given design $\rho \in \mathcal{D}$.
\label{remark:nested_form}
\end{remark}

The next section describes the filtering and threshold projection procedures on the design variable.
Although they are well addressed in the literature, here we summarize the key ingredients concerning concepts and solution methods that we consider useful to be presented in the same notation and context.

\subsection{Helmholtz-type PDE filtering and threshold projection}
\label{sec:helmholtz_threshold}

The filtered design field $\hat{\rho}$ is represented implicitly by the solution of the isotropic Helmholtz-type partial differential equation (PDE) \cite{lazarov2011filters}:
\begin{equation}
    \begin{array}{ll}
        -r^2 \nabla^2 \hat{\rho} + \hat{\rho} = \rho & \textmd{in} \ \Omega,  \\
        \nabla \hat{\rho} \cdot \bm{n} = 0 & \textmd{on} \ \Gamma, 
    \end{array}
    \label{eq:isotropic_helmholtz_pde_filtering}
\end{equation}
\noindent
where $\bm{n}$ is the outward normal of $\Omega$. Here $r$ is related to the length parameter that represents the usual density filtering radius size $\bar{r}$, by $r = \bar{r}/(2 \sqrt{3})$. Although not considered in this work, boundary effects can be eliminated by adopting the procedure described in \cite{clausen2017filter, wallin2020consistent}.

The variational formulation of \eqref{eq:isotropic_helmholtz_pde_filtering} is stated as: find $\hat{\rho} \in H^1(\Omega)$ such that
\begin{equation}
    a(\hat{\rho}, \eta) = L(\eta) \quad \forall \eta \in H^1(\Omega),
    \label{eq:variational_formulation_FEM}
\end{equation}

\noindent
with
\begin{equation}
    a(\hat{\rho}, \eta) 
    = \int_{\Omega} r^2 \nabla \hat{\rho} \cdot \nabla \eta \, \mathrm{d}x
    + \int_{\Omega} \hat{\rho} \eta \, \mathrm{d}x
    \quad \textmd{and} \quad
    L(\eta) = \int_{\Omega} \rho \eta \, \mathrm{d}x.
    \label{eq:varfhelmholtz}
\end{equation}
Formulation \eqref{eq:variational_formulation_FEM} has a unique solution $\hat{\rho} = \hat{\rho}(\rho)$ for every $\rho \in L^\infty(\Omega)$, guaranteed by the Lax-Milgram lemma, and $\hat{\rho}\in[0,1]$ for every admissible $\rho$, as a consequence of a maximum principle (c.f.~\cite[Proposition 9.30]{brezis2011functional}).
From the discrete point of view, however, e.g. when \eqref{eq:variational_formulation_FEM} is solved by the classical Galerkin Finite Element Method (FEM) based on continuous Lagrangian elements, the approximate solution for $\hat{\rho}$ may lie outside the range $[0,1]$, unless the mesh is fine enough. 
In our implementation, this issue is handled by applying $\min-\max$ operators when 
such a situation is detected. 
Finally, if discontinuous approximations for $\hat{\rho}$ are sought, alternative formulations may be used, such as finite volume methods \cite{eymard2000finite}, mixed methods \cite{brezzi2005mixed} or discontinuous Galerkin methods \cite{cockburn2012discontinuous, brezzi2000discontinuous, Massing2017DG}.

An important and immediate result concerning the pointwise convergence of the filtered design used in the next section is (see appendix \ref{sec:app_theorem1} for a proof):
\begin{theorem}
\label{thm:pointwisefilter} 
Let $(\rho_{n})$, $\rho_{n} \in L^{\infty}(\Omega,[0,1])$, be a sequence converging weakly$^*$ to some $\rho  \in L^{\infty}(\Omega,[0,1])$. Then the sequence of filtered designs $(\hat{\rho}(\rho_{n}))$ %
originated by the solutions of  \eqref{eq:variational_formulation_FEM}-\eqref{eq:varfhelmholtz} converges pointwise a.e.\ to $\hat{\rho}(\rho)$.
\end{theorem}

Notice that the density field should exhibit, ideally, binary values in $\{0, 1\}$. 
In our implementation,
intermediate densities are reduced by adopting a threshold projection that provides density fields with a steeper transition between $0$ and $1$.
In such cases,  the $\hat{\rho}$ values are projected to their maximum and minimum limits $\{0, 1\}$.
According to \cite{wang2011projection} this can
be accomplished by computing a new field $\bar{\rho}$ defined as:
\begin{equation*}
    \bar{\rho} (\hat{\rho}(\bm{x})) 
    = \frac{\tanh{\left( \beta \gamma \right)}
    +\tanh{\left( \beta(\hat{\rho}(\bm{x})-\gamma) \right)}}
    {\tanh{\left( \beta \gamma \right)}
    +\tanh{\left( \beta(1-\gamma) \right)}},
    \label{eq:threshold_projection}
\end{equation*}
where $\beta$ is the projection parameter responsible for weighting the intermediate density values and $\gamma$ translates the projection. Herein, we assume an increasing value of $\beta$ from $\beta_{\min}$ to $\beta_{\max}$, doubling it every $\beta_{d}$ iterations.

\section{General framework for TO with AM}
\label{sec:archetype_problem}

The standard density-based TO formulation \eqref{eq:standad_formulation_TO} is now extended to account for the physics of the AM process,
leading to a quite general formulation for TO aimed at AM which we refer to as the archetype problem. This is done by including additional terms in the objective function which takes into account properties of partially built structures. The general framework is then specialized in Sections \ref{sec:self_weight_based_formulation} and \ref{sec:thermal_conductivity_based_formulation} to include the mechanical and thermal compliances of such partially built structures as a way of controlling overhang.
{\color{blue}
Contrary to many methods developed in the literature \cite{qian2017undercut, allaire2018optimizing, thore2019penalty, mezzadri2020second}, the proposed formulation limits the overhang relative to the build plate without imposing a critical overhang angle.}

We consider an AM process occurring on the time interval $(0, T]$, where $T$ is the total process time. The domain and boundaries of interest of a partially built structure at time $t \in (0,T]$, are given by
\begin{equation*}
    \Omega_{t} = 
    \left\{ 
    \bm{x} \in \Omega \ \vert \ 0 < \bm{x} \cdot \bm{b} < H t/T
    \right\},
    \quad
    \Gamma_{t}^{u} = 
    \Gamma_t
    \setminus
    \Gamma
    \quad \textmd{and} \quad
    \Gamma_{0} = 
    \left\{ 
    \bm{x} \in \Gamma \ \vert \ \bm{x} \cdot \bm{b} = 0
    \right\},
\end{equation*}
where $\bm{b}$ is a unit vector that defines the build direction normal to the build plate, $H=\max \left( \bm{x} \cdot \bm{b} \right)$, $\bm{x} \in \Omega$, is the ``height'' of the design domain along the build direction, 
{\color{blue}
$\Gamma_t$ is the boundary of $\Omega_{t}$,
and $\Gamma_0$ represents the base (build plate).
The product $Ht/T$ in the definition of $\Omega_t$ refers only to a parameterization of the height $H$ as $t$ increases.}
In the AM process, the part $\Omega_t^S$ of $\Omega^S$ that has been created up to time $t$ is defined by ${\left. \rho \right|}_{\Omega_t}\left( \bm{x} \right) = 1$ (see Figure \ref{fig:archetype_problem_new2}), where ${\left. \rho \right|}_{\Omega_t}\left( \bm{x} \right) = {\left. \rho \left( \bm{x} \right) \right|}_{\bm{x} \in \Omega_t}$.
Note that the density field $\rho$ is simply $\rho = {\left. \rho \right|}_{\Omega_T}$.

\begin{figure}[H]
\centering
\includegraphics[scale=0.85]{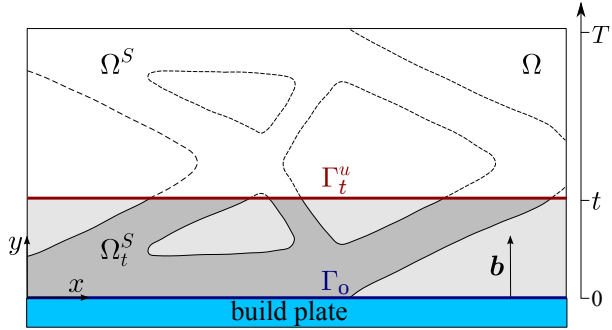}
\caption{Schematic illustration of the AM process.}
\label{fig:archetype_problem_new2}
\end{figure}

Denoting by $\mathcal{D} \subset L^{\infty}(\Omega)$ the set of admissible designs defined by the box  and volume constraints in \eqref{eq:standad_formulation_TO}, the archetype problem is defined as
\begin{equation}
    \begin{array}{ll}
      \left( \mathbb{AP} \right) 
      & \displaystyle \min_{\rho \in \mathcal{D}}
      \left[
        a \left( \rho; \bm{u}(\rho), \bm{u}(\rho) \right) 
        + \displaystyle \int_0^T \gamma(t) J_P \left( {\left. \rho \right|}_{\Omega_t}, \psi_P \left({\left. \rho \right|}_{\Omega_t}, t \right), t \right) \, \mathrm{d}t
        \right],
    \end{array}
    \label{eq:AP_problem}
\end{equation}
\noindent
where $\bm{u}\left( \rho \right)$ solves the state problem in \eqref{eq:standad_formulation_TO}.
The total cost function contains a cost for the design and a cost $J_P$ for each partially manufactured design $\Omega^S_t$. 
The latter may depend on temperature, displacement, or velocity fields, denoted collectively by $\psi_P = \psi_P \left( {\left. \rho \right|}_{\Omega_t}, t \right)$ which is the solution to a state problem posed on the domain $\Omega_{t}$ for a given design. 
A suitable, non-restrictive technical assumption is that $\psi_P \left({\left. \rho \right|}_{\Omega_t}, t \right) \in H \left(\Omega_t, \mathbb{R}^n \right)$ where $H \left( \Omega_t, \mathbb{R}^n \right)$ is a Banach space with norm ${\left\| \cdot \right\|}_{H(\Omega_{t})}$.

The non-negative (bounded and continuous) function $\gamma$ in \eqref{eq:AP_problem} can be used to balance the costs for the final and partially manufactured designs. With appropriately chosen costs, an optimal solution to this multi-objective problem defines a Pareto optimum \cite[Proposition 3, p.~297]{Aubin:1979}. Although not addressed in this work, a possible alternative formulation of this multi-objective problem is to convert some of the objective terms into constraints.

To solve problem \eqref{eq:AP_problem} numerically, {\color{blue} we must discretize the process in time and space.}
Towards this end, let $0 = t_0 < t_1 < \cdots < t_l = T$, be a partition of the interval $(0, T]$ into subintervals $I_i=(t_{i-1},t_{i}]$, $i=1, 2, \cdots, l$, such that
\begin{equation*}
    \lim_{l \to \infty}
    \left(
    \max_{i=1,2,\cdots, l} 
    |I_i|
    \right) = 0.
\end{equation*}
where $l$ is the number of discrete layers used to mimic the build process.
Figure \ref{fig:process_domain_boundaries_new2} illustrates the domains and boundaries that arise at the discrete time steps. By applying 
%the backward Euler method, 
a first-order backward approximation to the time integral, 
the semi-discrete version of the $\mathbb{AP}$ becomes
\begin{equation}
      \displaystyle \min_{\rho \in \mathcal{D}}
      \left[
        a \left( \rho; \bm{u}(\rho), \bm{u}(\rho) \right)
        + \sum_{i=1}^l w_i \,
        J_P \left( {\left. \rho \right|}_{\Omega_i}, \psi_P \left( {\left. \rho \right|}_{\Omega_i}, t_i \right), t_i \right) \right],
    \label{eq:AP_problem_new_semi_discrete}
\end{equation}
where $w_i = (t_i-t_{i-1}) \gamma(t_i)$ and $\Omega_i = \Omega_{t_i}$ for $i=1,2,\cdots, l$.
Discretization is completed by spatial discretization using the  Finite Element Method (FEM).

\begin{remark}
As noted by \cite{allaire2017structural, amir2018topology}, the time steps do not necessarily represent just one physical AM layer. Additionally, it is not possible to see how the material is being deposited along the i-th layer. Although not addressed in this study, such limitation may be overcome by adding contributions concerning the partially built structures along with the perpendicular direction to  $\bm{b}$.
\end{remark}

\begin{figure}[H]
\centering
\includegraphics[scale=0.85]{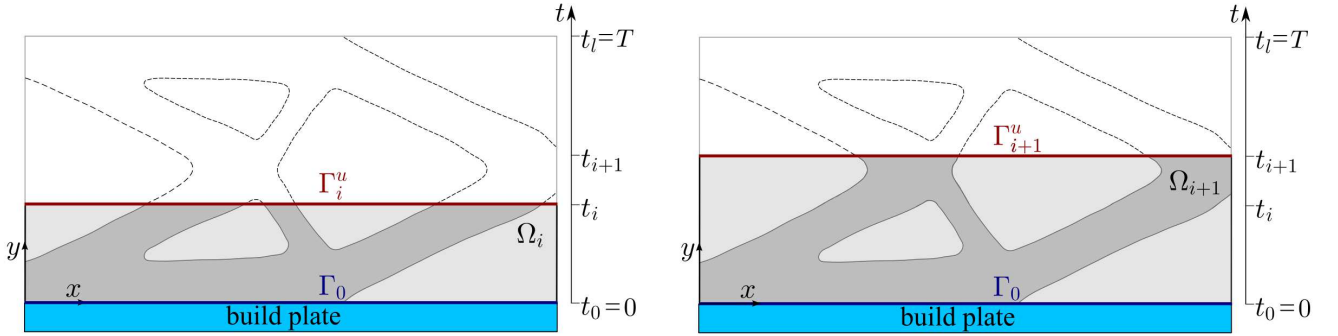}
\caption{Domains defined along the discrete process and its boundaries.}
\label{fig:process_domain_boundaries_new2}
\end{figure}

An important aspect that deserves attention is related to the convergence of solutions to problem \eqref{eq:AP_problem_new_semi_discrete} as the number of layers increases. 
According to Theorem \ref{thm:abstractconvergence} below, any such sequence will converge to a solution to problem \eqref{eq:AP_problem} (whose existence is an outcome of the theorem). In the statement of the theorem and its proof we compare functions defined on different domains and therefore we introduce the following definition of weak convergence (c.f \cite{haslinger1996finite}):

\begin{definition} 
Consider a sequence $\left( t_{i} \right)$ converging to some $t$. 
A sequence of functions $\left( v_{i} \right)$ such that $v_{i} \in H \left( \Omega_{i},\mathbb{R}^{n} \right)$ for all $i$ is said to converge weakly (indicated by $\rightharpoonup$) to a function $v \in H \left( \Omega_{t},\mathbb{R}^{n} \right)$ if $\left( \tilde{v}_{i} \right)$ converges weakly in $H \left( \Omega,\mathbb{R}^{n} \right)$ to $\tilde{v} \in H \left( \Omega,\mathbb{R}^{n} \right)$, where the $\tilde{v}$ denotes the extension of a function $v \in H \left( \Omega_{t},\mathbb{R}^{n} \right)$ to $H \left( \Omega,\mathbb{R}^{n} \right)$.
\end{definition}

\noindent Whether a function in $H \left( \Omega_{t},\mathbb{R}^{n} \right)$ can be extended to a function in $H \left(\Omega,\mathbb{R}^{n} \right)$ depends on properties of the domain in general; for the standard Sobolev spaces ${\left[ H^{1}\left( \omega \right) \right]}^{d}$ this is true for example for domains with Lipschitz boundary \cite{jones1981quasiconformal}.

\begin{theorem}
\label{thm:abstractconvergence}
Assume that \eqref{eq:AP_problem_new_semi_discrete} has at least one solution for every $l \geq 1$ and that every sequence of admissible designs $(\rho_{i}) \overset{\ast}{\rightharpoonup} \rho \in \mathcal{D}$ and times $(t_{i}) \to t$ has a subsequence such that the corresponding states converges weakly, i.e.\ $\psi_P({\left. \rho_{i} \right|}_{\Omega_{{i}}}, t_{i})  \rightharpoonup \psi_P ({\left. \rho \right|}_{\Omega_t}, t)$,
and satisfy $||\psi_P ( {\left. \rho_l \right|}_{\Omega_{{i}}}, t_{i} )||_{H(\Omega_{{i}})} \leq C$ with $C$ independent on $\Omega_{{i}}$. 
Furthermore, assume that for a sequence of the just described type, $J_{P}$ is continuous in the sense that $J_{P}(\rho_{i}, \psi_P ({\left. \rho_{i} \right|}_{\Omega_{{i}}} , t_{i} ), t_{i} ) \rightarrow J_{P} (\rho, \psi_P ({\left. \rho \right|}_{\Omega_t}, t ), t )$ point-wise a.e.\ in $(0,T]$. 
Finally assume that $J_{P}$ satisfies, with $c$ a positive constant,  $\left| J_{P}(\rho,\psi_P({\left. \rho \right|}_{\Omega_t}, t), t) \right|
\leq c
{\left\| \psi_P({\left. \rho \right|}_{\Omega_t}, t)\right\|}_{H(\Omega_{t})}$ for every $t \in (0,T]$. Then every sequence of solutions to \eqref{eq:AP_problem_new_semi_discrete} has a subsequence converging weakly$^*$ to a solution to \eqref{eq:AP_problem} as $l \rightarrow \infty$.
\label{teo:convergence}
\end{theorem}
\begin{proof}
Consider a sequence of solutions $\left(\rho_l^*\right)$ to the semi-discretized problem \eqref{eq:AP_problem_new_semi_discrete}. 
By definition of a minimizer there holds that
\begin{multline}
        a \left( \rho_l^*; \bm{u} \left( \rho_l^* \right), \bm{u} \left( \rho_l^* \right) \right)
        + \displaystyle \int_0^T \sum_{i=1}^{l} \gamma \left( t_{i} \right) J_P \left( {\left. \rho_l^* \right|}_{\Omega_i}, \psi_P \left({\left. \rho_l^* \right|}_{\Omega_{i}}, t_i \right), t_i \right) \chi_{i}(t) \, \mathrm{d}t \leq \\
        a \left( \rho; \bm{u}\left(\rho\right), \bm{u}\left(\rho\right)\right)
        + \displaystyle \int_0^T \sum_{i=1}^{l} \gamma \left( t_{i} \right) J_P \left( {\left. \rho \right|}_{\Omega_{{i}}}, \psi_P \left( {\left. \rho \right|}_{\Omega_{{i}}}, t_i \right), t_i \right) \chi_{i} \left( t \right) \, \mathrm{d}t \quad \forall \rho \in \mathcal{D},
\label{eq:defminimizer}
\end{multline}
where the indicator function $\chi_{i} \left( t \right) = 1$ if $t \in I_{i}$ and $\chi_{i} \left( t \right) = 0$ otherwise. 
Since $\mathcal{D}$ is weakly$^*$ sequentially compact \cite{borrvall2001topology}, we may extract a subsequence, again denoted $\left( \rho_{l}^{*} \right)$, converging weakly$^*$ to some $\rho^* \in \mathcal{D}$. 
Since $a \left( \rho; \cdot, \cdot \right)$ is coercive for every $\rho \in \mathcal{D}$ and $L$ is continuous we deduce that ${\left\| \bm{u} \left( \rho_{l}^{*} \right) \right\|}_{1, \Omega} \leq C$ for all $l$ and some constant $C$.
The sequence of states is thus bounded and we may extract a further subsequence $\left( \rho_{l_m}^{*} \right)$ of designs such that the corresponding sequence of states converges weakly (in ${[H^{1}(\Omega)]}^{d}$) to some $\bm{u} \in V$. Using the pointwise convergence of the filtered design (Theorem \ref{thm:pointwisefilter}) one can show (c.f.~e.g~\cite{petersson1999some}) that this $\bm{u}$ satisfies $a \left( \rho^*; \bm{u}, \bm{v} \right) = L \left( \bm{v} \right)$ for all $\bm{v} \in V$ , i.e.\ $\bm{u} = \bm{u}\left( \rho^* \right)$. 
We have thus obtained a subsequence of designs such that the first term on the left in \eqref{eq:defminimizer} converges to $a \left( \rho^*; \bm{u} \left(\rho^* \right), \bm{u} \left( \rho^* \right) \right)$.

To conclude the proof, we show that for any sequence $ \left( \rho_l \right)$ of admissible designs tending weakly$^*$ to some $\rho \in \mathcal{D}$ there is a subsequence such that 
\begin{equation}
    \displaystyle \int_0^T \sum_{i=1}^{l} \gamma \left( t_{i} \right) J_P \left( {\left. \rho_l \right|}_{\Omega_{i}}, \psi_P \left({\left. \rho_l \right|}_{\Omega_{{i}}}, t_{i} \right) ,t_{i} \right) \chi_{i}(t) \, \mathrm{d}t \rightarrow 
    \displaystyle \int_0^T \gamma \left( t \right) J_P 
    \left(
    {\left. \rho \right|}_{\Omega_t},
    \psi_P \left( {\left. \rho \right|}_{\Omega_t}, t  \right),
    t
    \right) \, \mathrm{d}t.
\label{eq:integralconvergence}
\end{equation}
Towards this end, let $\tau \in (0,T]$ be given. 
This $\tau$ belongs to some interval $I_{k}$, $k(l) \in \{1, \ldots, l\}$, so
\begin{multline*}
    \left|
    \sum_{i=1}^{l} \gamma \left( t_{i} \right)
    J_P \left(    
        {\left. \rho_l \right|}_{\Omega_{i}},
        \psi_P \left({\left. \rho_l \right|}_{\Omega_{{i}}}, t_{i} \right),
        t_{i}
        \right)
    \chi_{i}(\tau)
    - \gamma \left( t \right)
    J_P \left(
        {\left. \rho \right|}_{\Omega_t},
        \psi_P \left( {\left. \rho \right|}_{\Omega_t}, t \right),
        t
        \right)
    \right| = \\
    \left|
    \gamma \left( \tau_{k} \right)
    J_P \left(   
        {\left. \rho_l \right|}_{\Omega_{\tau_k}},
        \psi_P \left( {\left. \rho_l \right|}_{\Omega_{\tau_{k}}}, \tau_{k} \right),
        \tau_{k}
        \right)
    - \gamma \left( \tau \right)
    J_P \left(
        {\left. \rho \right|}_{\Omega_\tau},
        \psi_P \left( {\left. \rho \right|}_{\Omega_\tau}, \tau \right),
        \tau
        \right)
    \right|.
\end{multline*}
Since $\left| \tau_{k(l)}-\tau \right| \leq \max_{i=1,\ldots,l} \left|I_{i}\right| \to 0$ as $l \to \infty$, the sequence $(\tau_{k(l)})$ converges to $\tau$ as $l \to \infty$. By assumption we may extract a subsequence (not relabeled) of densities and times such that $\psi_P \left( {\left. \rho_l \right|}_{\Omega_{\tau_{k(l)}}}, \tau_{k(l)} \right)$ tends weakly to $\psi_P \left( {\left. \rho \right|}_{\Omega_{\tau}}, \tau \right)$. Then, from the assumed continuity of $\gamma$ and $J_{P}$ and the fact that $\tau \in (0,T]$ was arbitrary follows the point-wise convergence of the integrand on the left in \eqref{eq:integralconvergence} to that on the right. Finally, letting $c_{\gamma}$ denote the upper bound on $\gamma$ over $(0,T]$,
\begin{equation*}
    \left|
    \sum_{i=1}^{l} \gamma \left( t_{i} \right)
    J_P \left( 
        {\left. \rho_l \right|}_{\Omega_{i}}, 
        \psi_P \left({\left. \rho_l \right|}_{\Omega_{{i}}}, t_{i} \right),
        t_{i}
        \right)
    \chi_{i}(\tau)
    \right|
    \leq
    c_{\gamma} c 
    {\left\|
    \psi_P \left({\left. \rho_l \right|}_{\Omega_{\tau_{k(l)}}}, \tau_{k(l)} \right)
    \right\|}_{H \left( \Omega_{\tau_{k(l)}} \right)}
    \leq
c_{\gamma} c  C,
\end{equation*}
where the last inequality follows since $(\psi_P({\left. \rho_l \right|}_{\Omega_{\tau_{k(l)}}}, \tau_{k(l)} ))$ is a weakly convergent, hence bounded sequence. By assumption, $C$ is independent of $\Omega_{\tau_{k}}$, and since $c_{\gamma} c C$ is independent of $l$ and Lebesgue integrable on $(0,T]$, the convergence \eqref{eq:integralconvergence} follows from the Lebesgue dominated convergence theorem.

According to the preceding paragraph, we may now from $\left( \rho_{l_m}^{*} \right)$ extract a further subsequence such that the convergence \eqref{eq:integralconvergence} holds on both sides (with ``$\rho_l$'' = $\rho_l^*$ and ``$\rho_l$'' = $\rho$ for all $l$ respectively) in \eqref{eq:defminimizer}, and thus that in the limit,  
\begin{multline*}
        a \left( \rho^*; \bm{u} \left( \rho^* \right), \bm{u} \left( \rho^* \right) \right)
        + \displaystyle \int_0^T \gamma \left( t \right) J_P \left( {\left. \rho^* \right|}_{\Omega_t}, \psi_P \left({\left. \rho^* \right|}_{\Omega_{t}}, t \right), t \right) \, \mathrm{d}t \leq \\
        a \left( \rho; \bm{u}\left(\rho\right), \bm{u}\left(\rho\right)\right)
        + \displaystyle \int_0^T \gamma \left( t \right) J_P \left( {\left. \rho \right|}_{\Omega_t}, \psi_P \left( {\left. \rho \right|}_{\Omega_t}, t \right), t \right) \, \mathrm{d}t \quad \forall \rho \in \mathcal{D},
\end{multline*}
or in other words that $\rho^*$ solves \eqref{eq:AP_problem}.
\end{proof}

The fact \textit{every} sequence of solutions to \eqref{eq:AP_problem_new_semi_discrete} has a subsequence converging to a solution to \eqref{eq:AP_problem} as $l \rightarrow \infty$ means that every sequence converges to the set of solutions to \eqref{eq:AP_problem}. 
If \eqref{eq:AP_problem} has a unique solution, the entire sequence will thus converge to this solution. 
{\color{blue}Note also that, while} the designs only converge weakly$^*$, the associated filtered designs $\hat{\rho}$ converges point-wise (recall Theorem \ref{thm:pointwisefilter}).
An 
%interesting 
important aspect not addressed by Theorem \ref{thm:abstractconvergence} is the \textit{rate of} convergence. The numerical examples below suggest that it can be quite fast and it would be very interesting to obtain theoretical results in this direction. Appendix \ref{sec:app_a} provides an example of how to verify the assumptions of Theorem \ref{thm:abstractconvergence} in a practical case.

As previously stated, the cost $J_P$ of each partially manufactured design may be functions of $\psi_P$. 
In the following, we present two particularizations for the $\psi_P$ dependence based on self-weight and thermal conductivity physics.

\subsection{Self-weight-based formulation}
\label{sec:self_weight_based_formulation}

This section presents the self-weight-based formulation as the first alternative to reduce overhang regions.
In the total cost function given in \eqref{eq:AP_problem_new_semi_discrete}, the final design cost, $J_D$, is defined by the compliance of a structure subject to the fixed external loads, whereas the design cost, $J_P$, is represented by the compliance of partially manufactured structures subject to self-weight loading.

By considering the compliance-based TO \eqref{eq:standad_formulation_TO} and the discrete version of $\mathbb{AP}$ \eqref{eq:AP_problem_new_semi_discrete}, the self-weight-based formulation is stated as: 
Find a density scalar field $\rho \in \mathcal{D}$ and displacement fields $\bm{u} \in V$, $\bm{u}_i \in V_i$ for $i=1,2,\cdots, l$, such that
\begin{equation}
    \begin{array}{cll}
        \displaystyle 
        \min_{\rho \in \mathcal{D}}
        & a(\rho; \bm{u}, \bm{u}) + \displaystyle \sum_{i=1}^{l} w_i \ a_{i}(\rho; \bm{u}_i,\bm{u}_i) & \textmd{[total cost]} \\
        \textmd{s.t.} 
& a(\rho; \bm{u},\bm{v}) = L(\bm{v}) \quad \forall \bm{v} \in {V}  & \textmd{[balance equation: main problem]}\\
& a_i(\rho; \bm{u}_i,\bm{v}_i) = L_i(\rho; \bm{v}_i) \quad \forall \bm{v}_i \in {V}_{i}, \ i = 1, 2, \ldots, l & \textmd{[balance equations: sub-problems]}\\
    \end{array}
    \label{eq:self_weight_formulation_TO}
\end{equation}
with trial and test spaces $V_{i}$ equally defined as
\begin{equation*}
    V_{i} = \left\{ \bm{v}_i \in {\left[ H^1(\Omega_i) \right]}^d \ \vert \ \bm{v}_i = \bm{0} \ \textmd{on} \ \Gamma_{0} \right\} 
    \quad \textmd{for} \ i = 1, 2, \ldots, l
    .
\end{equation*}
The corresponding bilinear and linear forms are defined by \eqref{eq:varastandard} and
\begin{equation}
    a_i(\rho; \bm{u}, \bm{v}) = \int_{\Omega_i} \bm{\sigma}^\mathrm{RAMP} \left( {\left. \rho \right\vert}_{\Omega_i}, \bm{u} \right) : \bm{\varepsilon} (\bm{v}) \ud x
    \quad \textmd{and} \quad
    L_i(\rho; \bm{v}) = \int_{\Omega_i} \bm{f}\left( {\left. \rho \right\vert}_{\Omega_i} \right) \cdot \bm{v} \ud x 
    \quad \textmd{for} \ i = 1, 2, \ldots, l .
    \label{eq:bilinear_linear_form_subproblems_self_weight}
\end{equation}
\noindent
where $\bm{f}\left( {\left. \rho \right\vert}_{\Omega_i} \right) \in [L^{2}(\Omega_i)]^d$ is the body force per unit volume. As shown in Appendix \ref{sec:app_a}, this formulation verifies the assumptions of Theorem \ref{teo:convergence}.

{\color{blue}
\paragraph{Sensitivity analysis}
Let us derive the sensitivity of the objective function with respect to the physical field $\bar{\rho}$.
The sensitivity of the physical field $\bar{\rho}$, with respect to the optimization variable $\rho$, can be easily evaluated by differentiating the filtering equations and, therefore, will not be described here (see \cite{qian2013topological} for details).
The Lagrangian $L_{\textmd{sw}}$ for the self-weight-based objective function (\ref{eq:self_weight_formulation_TO}a) is defined as
\begin{equation*}
    L_{\textmd{sw}}(\rho, \bm{u}, \bm{u}_i, \bm{\lambda}, \bm{\mu}_i) := 
    a(\rho; \bm{u}, \bm{u}) 
    + \displaystyle \sum_{i=1}^{l} 
    w_i a_{i}(\rho; \bm{u}_i,\bm{u}_i) 
    + a(\rho; \bm{u}, \bm{\lambda}) 
    - L(\bm{\lambda}) + \displaystyle \sum_{i=1}^{l} \left[ a_i(\rho; \bm{u}_i,\bm{\mu}_i) - L_i(\rho; \bm{\mu}_i) \right],
\end{equation*}

\noindent
where $\bm{\lambda} \in V$ and $\bm{\mu}_i \in V_i$, $i=1, \cdots, l$, are the Lagrange multipliers associated with the balance equations (\ref{eq:self_weight_formulation_TO}b) and (\ref{eq:self_weight_formulation_TO}c), respectively.
Note that $L_{\textmd{sw}}$ has the same value as the objective function for all $\rho$ since the equations of state are satisfied; hence the derivatives of $L_{\textmd{sw}}$ and those of the objective function will coincide.

The Gâteaux derivative of $L_{\textmd{sw}}$ with respect to the physical field $\bar{\rho}$ evaluated by the adjoint sensitivity method is given by
\begin{eqnarray*}
    \frac{\ud L_{\textmd{sw}}}{\ud \bar{\rho}} 
    &=&
    \frac{\partial L_{\textmd{sw}}}{\partial \bar{\rho}}
    +
    \frac{\partial}{\partial \bm{u}}
    \left[
    a(\rho; \bm{u}, \bm{u})
    + a(\rho; \bm{u}, \bm{\lambda})
    - L(\bm{\lambda})
    \right]
    \cdot
    \dfrac{\ud \bm{u}}{\ud \bar{\rho}} 
    \nonumber
    \\
    && 
    \quad +
    \sum_{i=1}^{l}
    \frac{\partial}{\partial \bm{u}_i}
    \left[
    w_i a_{i}(\rho; \bm{u}_i,\bm{u}_i)
    + a_i(\rho; \bm{u}_i,\bm{\mu}_i) 
    - L_i(\rho; \bm{\mu}_i)
    \right]
    \cdot
    \dfrac{\ud \bm{u}_i}{\ud \bar{\rho}},
    \label{eq_Lsw_derivative}
\end{eqnarray*}
where we see that computation of the complicated terms ${\ud \bm{u}}/{\ud \bar{\rho}}$ and ${\ud \bm{u}_i}/{\ud \bar{\rho}}$ can be avoided if the multipliers are chosen as $\bm{\lambda}=-2\bm{u}$ and $\bm{\mu}_i=-2 w_i \bm{u}_i$, respectively. The derivative then reduces to
\begin{equation*}
    \frac{\ud L_{\textmd{sw}}}{\ud \bar{\rho}} =
    \frac{\partial}{\partial \bar{\rho}} \left( -a(\rho; \bm{u}, \bm{u}) + \displaystyle \sum_{i=1}^{l} w_i \left[ 2 L_i(\rho; \bm{u}_i) - a_{i}(\rho; \bm{u}_i,\bm{u}_i) \right] \right),
\end{equation*}

\noindent
where the partial derivatives are given, for each $i = 1, 2, \ldots, l$, by
\begin{equation}
    \label{eq:first_rhs_partial_derivative}
    \frac{\partial}{\partial \bar{\rho}} a(\rho; \bm{u}, \bm{u}) = 
    \int_{\Omega} 
    E' \left( \rho \right) \delta \bar{\rho} \ 
    \bm{\varepsilon} (\bm{u}) : \mathbb{D}_0( \nu ) : \bm{\varepsilon} (\bm{u}) \, \mathrm{d}x,
\end{equation}
\begin{equation*}
    \frac{\partial}{\partial \bar{\rho}} a_{i}(\rho; \bm{u}_i,\bm{u}_i) = 
    \int_{\Omega_i} 
    {E'}^\mathrm{RAMP} \left( {\left. \rho \right\vert}_{\Omega_i} \right) \delta \bar{\rho} \ \bm{\varepsilon} (\bm{u}_i) : \mathbb{D}_0( \nu ) : \bm{\varepsilon} (\bm{u}_i) \, \mathrm{d}x
    \quad \textmd{and} \quad
    \frac{\partial}{\partial \bar{\rho}} L_i(\rho; \bm{u}_i) = 
    \int_{\Omega_i} \bm{f}'\left( {\left. \rho \right\vert}_{\Omega_i} \right) \delta \bar{\rho} \cdot \bm{u}_i \ud x,
\end{equation*}

\noindent
with Young's modulus derivatives
\begin{equation*}
    E' \left( \bar{\rho}(\rho) \right) = q \bar{\rho}^{q-1} \left( E_0 - E_{\min} \right) 
    \quad \textmd{and} \quad
    {E'}^\mathrm{RAMP} \left( \bar{\rho}(\rho) \right) = \frac{\left(1+q \right)}{{\left[1+q(1-\bar{\rho})\right]}^2} \left( E_0 - E_{\min} \right),
\end{equation*}
and $\delta \bar{\rho}$ denoting the variation of the field $\bar{\rho}$.
For the sake of convenience of exposition, in the previous expressions we considered the stress tensor $\bm{\sigma}$ given in \eqref{eq:stress_tensor_expression} rewritten by using the identity $\mathbb{D} \left( \lambda, \mu \right) = E \left( \rho \right) \mathbb{D}_0( \nu )$ for an appropriate fourth-order tensor $\mathbb{D}_0$.}

\subsection{Thermal conductivity-based formulation}
\label{sec:thermal_conductivity_based_formulation}

The thermal conductivity-based formulation considers the temperature gradient between the build plate and the current manufactured layer, trying to mimic the heat deposition into the structure over the building process.
The design of each partially manufactured structure will contain channels that may be seen as supports for an optimal heat transfer. 
The relation between high thermal gradients and deformations has been treated in the literature.
In \cite{CHENG2015102}, the authors apply a thermomechanical model to overhang structures in the EBM process to evaluate the temperature induced deformation  and compare with experiments. They claim (page 103) that the heat load, more than the gravity, is the source of the deformation. This is also illustrated by printed parts where the warping is upwards and not downwards as would be predicted by self-weight. This is further developed in \cite{Cooper2018} where they pursue the idea of contact free supports that serve as heat sinks for the overhang structures. For laser printing the problem of residual stresses is well known. The paper \cite{Patterson2017} offers a review of the residual stress problem and its relation to thermal deformation for overhang structures for SLM.

Figure \ref{fig:motivation_heat_formulation} illustrates the temperature field over a generic partially built structure exhibiting an overhang region.
The localized high temperature on the partially built structure mimics a physical observation in the EBM process, which is induced by the melting of the material when a new layer is deposited. 
The difficulty in dissipating  such energy deposition may cause excessive deformations and warping defects, due to the residual thermal stresses.
The formulation of this section, based on thermal conductivity, aims to limit such behavior.

\begin{figure}[H]
     \centering
     \includegraphics[scale=0.85]{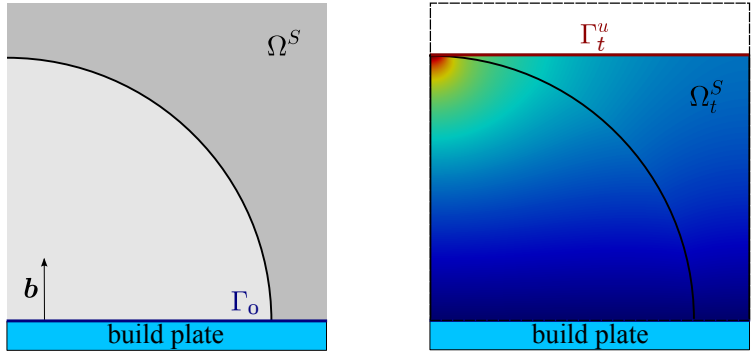}
     \caption{Density and temperature fields in a partially built structure.}
     \label{fig:motivation_heat_formulation}
\end{figure}

The thermal conductivity-based formulation makes use of the steady-state heat-conduction problem, and is stated as:
Find a density scalar field $\rho \in \mathcal{D}$, displacement field $\bm{u} \in V$ and temperature fields $\theta_i \in U_i$, for $i=1,2,\cdots, l$, such that
\begin{equation}
    \begin{array}{cll}
        \displaystyle 
        \min_{\rho \in \mathcal{D}}
        &  a(\rho; \bm{u}, \bm{u}) + \displaystyle \sum_{i=1}^{l} w_i \ b_i(\rho; \theta_i,\theta_i) & \textmd{[total cost]} \\
        \textmd{s.t.} 
& a(\rho; \bm{u},\bm{v}) = L(\bm{v}) \quad \forall \bm{v} \in {V}  & \textmd{[balance equation: main problem]}\\
&  b_i(\rho; \theta_i, \eta_i) = l_i(\rho; \eta_i) \quad \forall \eta_i \in {U}_{i}, \ i = 1, 2, \ldots, l &  \textmd{[energy balance: sub-problems]} 
    \end{array}
    \label{eq:heat_formulation_TO}
\end{equation}
with trial and test spaces $U_i$ equally defined as
\begin{equation*}
    U_i = \left\{ \eta_i \in H^1(\Omega_i) \ | \ \eta_i = 0 \ \textmd{on} \ \Gamma_{0} \right\}, 
    \quad \ i = 1, 2, \ldots, l
    .
\end{equation*}
The above bilinear and linear forms are given according to \eqref{eq:varastandard} and
\begin{equation*}
	\refstepcounter{equation} 
	\latexlabel{eq:bilinear_linear_form_subproblems_heat_self_weight_a}
	\refstepcounter{equation} 
	\latexlabel{eq:bilinear_linear_form_subproblems_heat_self_weight_b}
    b_i(\rho; \theta, \eta) = \int_{\Omega_i} 
    {\color{blue}
    \kappa \left({\left. \rho \right\vert}_{\Omega_i} \right)}
    \ \nabla \theta \cdot \nabla \eta \ud x
    \quad \textmd{and} \quad
    l_i(\rho;\eta) = \int_{\Gamma_{i}^u} 
    {\color{blue}
    {\left. \bar{\rho} (\rho) \right\vert}_{\Omega_i}} q_i \eta \ud s
    \quad \textmd{for} \ i = 1, 2, \ldots, l .
\end{equation*}
{\color{blue}Herein, $\kappa$ is a pseudo thermal conductivity given by
\begin{equation*}
    \kappa (\bar{\rho}\left( \rho \right)) = 
    \kappa_{\min} + \frac{\bar{\rho}}{1+q(1-\bar{\rho})} \left( 1 - \kappa_{\min} \right),
\end{equation*}
where $\kappa_{\min}$}, taken as $10^{-9}$ in the numerical experiments, is a small positive parameter introduced for numerical purposes, {\color{blue} and} $q_i$ is the source of heat at the $i$-th layer.
From now on, we adopt for the penalty factor $q$ the symbol $q_m$ for the main problem and $q_s$ for the sub-problems. 
That the thermal conductivity-based formulation verifies the assumptions of Theorem \ref{teo:convergence} can be shown using similar techniques as for the self-weight-based formulation, employing a uniform trace inequality \cite{boulkhemair2008uniform} to handle the boundary integral in the sub-problems.

{\color{blue}
\paragraph{Sensitivity analysis}
The sensitivity of the objective function with respect to the physical density field $\bar{\rho}$ is obtained using the adjoint method as described in Section \ref{sec:self_weight_based_formulation}.
Let $L_{\textmd{tc}} = L_{\textmd{tc}}(\rho, \bm{u}, \theta_i, \bm{\lambda}, \mu_i)$ be the Lagrangian for the thermal conductivity-based objective function
and the  Lagrange multipliers
 $\bm{\lambda} = -2\bm{u}$ and $\mu_i=-2 w_i \theta_i$, $i=1, \cdots, l$.
By similar arguments, the Gâteaux derivative of $L_{\textmd{tc}}$ with respect to the physical field $\bar{\rho}$ becomes
\begin{equation}
    \frac{\ud L_{\textmd{tc}}}{\ud \bar{\rho}} =
    \frac{\partial}{\partial \bar{\rho}} \left( -a(\rho; \bm{u}, \bm{u}) + \displaystyle \sum_{i=1}^{l} w_i \left[ 2 l_i(\rho; \theta_i) - b_{i}(\rho; \theta_i, \theta_i) \right] \right),
    \label{eq:derivative_thermal_problem}
\end{equation}

\noindent
where the partial derivatives are given by
\begin{equation*}
    \frac{\partial}{\partial \bar{\rho}} b_{i}(\rho; \theta_i, \theta_i)
    = \int_{\Omega_i} 
    \kappa' \left({\left. \rho \right\vert}_{\Omega_i} \right) \delta \bar{\rho}
    \ \nabla \theta_i \cdot \nabla \theta_i \ud x
    \quad \textmd{and} \quad
    \frac{\partial}{\partial \bar{\rho}} l_i(\rho; \theta_i)
    = \int_{\Gamma_{i}^u} 
    \delta \bar{\rho} \ q_i \theta_i \ud s
    \quad \textmd{for} \ i = 1, 2, \ldots, l ,
\end{equation*}

\noindent
with
\begin{equation*}
\kappa ' (\bar{\rho}\left( \rho \right)) =
    \frac{\left(1+q \right)}{{\left[1+q(1-\bar{\rho})\right]}^2} \left( 1 - \kappa_{\min} \right).
\end{equation*}

\noindent
The derivative of the first term on the right-hand side of \eqref{eq:derivative_thermal_problem} is provided in \eqref{eq:first_rhs_partial_derivative}.}

\subsection{{\color{blue}A measure for the overhang surface}}
\label{sec:ameasure}

\noindent
In this study, we adopt as a quantitative measure of overhang the Projected Undercut Perimeter (PUP), that is defined according to \cite{qian2017undercut} as
\begin{equation}
    P_{\bar{\alpha}} 
    := \displaystyle \int_{\Omega}
       H \left( \bm{b} \cdot \dfrac{\nabla \bar{\rho}}{\| \nabla \bar{\rho} \|} - \cos{\left( \bar{\alpha} \right)} \right)
       \bm{b} \cdot \nabla \bar{\rho} \ \mathrm{d}x ,
    \label{eq:pup_definition}
\end{equation}
that quantify the projected perimeter of overhang regions defined by a critical overhang angle $\bar{\alpha}$. 
Herein, $H \left( \cdot \right)$ is the shifted continuous Heaviside projection function of the undercut perimeter
\begin{equation*}
    H \left ( \xi \right) =  
    \dfrac{1}{1+\exp{(-2 \zeta \xi)}},
    \label{eq:pup_heaviside}
\end{equation*}
with $\zeta$ controlling  the approximation to the discontinuous Heaviside function. 
In this work, we adopt the value $\zeta = 10$. 
Notice that in (\ref{eq:pup_definition}), $\nabla \bar{\rho}/\| \nabla \bar{\rho} \|$ estimates the normal to the diffuse interface between the full-material and void regions.
%pointing outwards to the latter.
We also consider the Normalized Projected Undercut Perimeter (NPUP), denoted by $NP_{\bar{\alpha}}$, that is defined as the ratio between $P_{\bar{\alpha}}$ and the area of the build plate, which provides a more appropriate measure of the overhang region.

{\color{blue}
For further comparison, we also introduce the projected perimeter based formulation for minimal overhang angle control described in \cite{qian2017undercut}, here simply identified as PUP formulation, which includes two additional constraints into the standard TO problem \eqref{eq:standad_formulation_TO} based on the PUP and grayness measures:
\begin{equation}
    \begin{array}{cl}
    P_{\bar{\alpha}} \leq \bar{P}_{\bar{\alpha}}
    & \textmd{[PUP constraint]} \\
    \displaystyle \int_{\Omega}
     \left.
     4 \bar{\rho} (1- \bar{\rho}) \ \mathrm{d}x \middle/ \int_{\Omega} \ \mathrm{d}x \right.
       \leq \bar{\epsilon} 
    & \textmd{[grayness constraint]} \\
    \end{array}
    \label{eq:pup_grayness_constraints}
\end{equation}
where $\bar{P}_{\bar{\alpha}}$ and $\bar{\epsilon}$ are the maximum allowed projected perimeter and grayness, respectively.
According to \cite{qian2017undercut}, the grayness constraint (\ref{eq:pup_grayness_constraints}b) prevents the appearance of trivial solutions of gray density with zero density gradient.
It is important to note that this formulation does not include the side-boundary-induced undercut constraint, that may be important to avoid not self-supporting drippings at domain boundaries \cite{qian2017undercut}.}

\section{Numerical examples}
\label{sec:num_results}

In this {\color{blue}section, we} perform topology optimization of 2D and 3D cantilever beams by adopting the proposed methodology with the self-weight formulation \eqref{eq:self_weight_formulation_TO} and the thermal conductivity formulation \eqref{eq:heat_formulation_TO}.
In Table \ref{table:general_parameters_cantilever_beam_problem}, we show the parameters that are common to all simulations. 
The remaining data, specific to each numerical experiment, are provided within the subsections.
{\color{blue}
We adopt the bilinear isoparametric quadrilateral elements for 2D cases and trilinear isoparametric hexahedron elements for 3D ones.}
Linear systems arising from the discretization of the main problem are solved by using the direct solver MUMPS \cite{amestoy2000multifrontal} in the 2D problems and by the PETSc \cite{petsc-web-page} preconditioned conjugate gradient solver in the 3D cases.
The corresponding linear systems arising from the sub-problems are solved by using MUMPS.
The stopping {\color{blue}criterion for the optimizer is based on the physical density variable of two consecutive iterations as}
\begin{equation*}
    {\| \bar{\rho}_{n}-\bar{\rho}_{n-1} \|}_{\infty} < \epsilon.
\end{equation*}
For convenience, we adopt {\color{blue}an} uniform time partition, and the penalty factor $\gamma$ in \eqref{eq:AP_problem} described in terms of the parameter $w_0 \in (0, 1]$ as $\gamma(t_i) = (1-w_0) / w_0$ for $i=1,2,\cdots, l$.
From these considerations, $w_i$ is particularized to $w_i = (T/l) (1-w_0) / w_0$.

The overall computational framework is written in Python using the open-source finite element framework FEniCS \cite{alnaes2015fenics} to set up and solve the state problems. As for the optimization method, we use the Method of Moving Asymptotes (MMA) proposed by Svanberg \cite{svanberg1987method}, whose Python code-version is available in  \cite{Deetman2020}. 
All codes necessary to reproduce the optimizations to be presented are available in  \cite{Haveroth2022}.
{\color{blue}
Finally, we remark that no numerical issues that could be associated with badly scaled objectives and or constraints were reported in the numerical experiments of this section.}

\begin{table}[H]
\centering
\begin{tabular}{p{4.6cm} p{1.6cm} p{1.2cm} | p{4.6cm} p{1.6cm} p{1.2cm}} 
 \hline
 \hline
 Parameter & Symbol & Value & Parameter & Symbol & Value \\
 \hline
 Load [$\mathrm{N}$] & $t$ & $1$ &
       Total process time [$\mathrm{s}$] & $T$ & $1$ \\
 Gravity [$\mathrm{m}/\mathrm{s}^2$] & $g$ & $9.81$ &
       Heat source [$\mathrm{K}$] & $q_i$ & $1$ \\
 Young's modulus [$\mathrm{Pa}$] & $E_0$ & $1$ &
       SIMP penalty & $q_m$ & $5$ \\
 Min. Young's modulus [$\mathrm{Pa}$] & $E_{\min}$ & $10^{-9}$  &
       RAMP penalty & $q_s$ & $5$ \\
 Poisson ratio & $\nu$ & $0.3$ &
       Tolerance & $\epsilon$ & $10^{-2}$ \\        
 Initial density [$\mathrm{Kg}/\mathrm{m}^3$] & $\rho_0$ & $\bar{v}$ &
 Min. conductivity [$\mathrm{W}/(\mathrm{m.K})$] & $k_{\min}$  & $10^{-9}$  \\
 \hline
 \hline
\end{tabular}
\caption{General parameters for all simulations.}
\label{table:general_parameters_cantilever_beam_problem}
\end{table}

\subsection{2D cantilever beam}
\label{sec:2d_cantilever_beam}

The 2D cantilever beam plane-strain problem is described in {\color{blue}Figure \ref{fig:short_cantilever_beam_geometry_boundary_conditions}}.
The left figure corresponds to the main problem, whereas the right one describes the boundary conditions and the  build-direction for the sub-problems.
Other relevant parameters used throughout this section are given in Table \ref{table:parameters_short_cantilever_beam_problem}.

\begin{figure}[H]
     \centering
    \begin{subfigure}[b]{0.45\textwidth}
         \centering
         \includegraphics[scale=1.00]{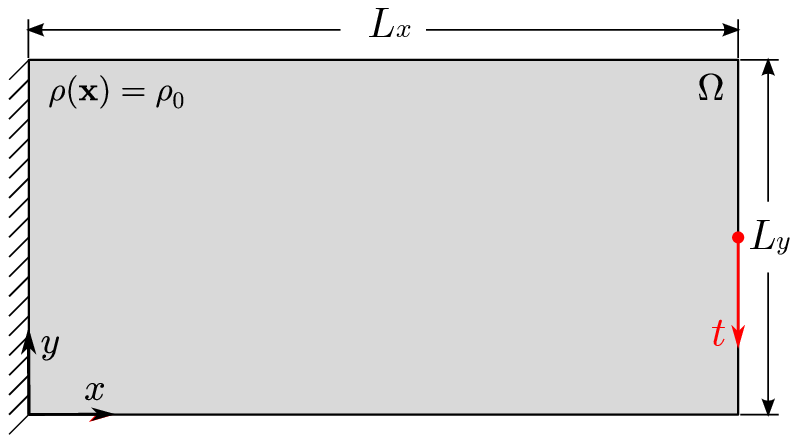}
         \caption{Main problem conditions.}
         \label{fig:short_cantilever_beam_geometry_boundary_conditions_main}
     \end{subfigure}
     \hspace{5pt}
     \begin{subfigure}[b]{0.45\textwidth}
         \centering
         \includegraphics[scale=1.00]{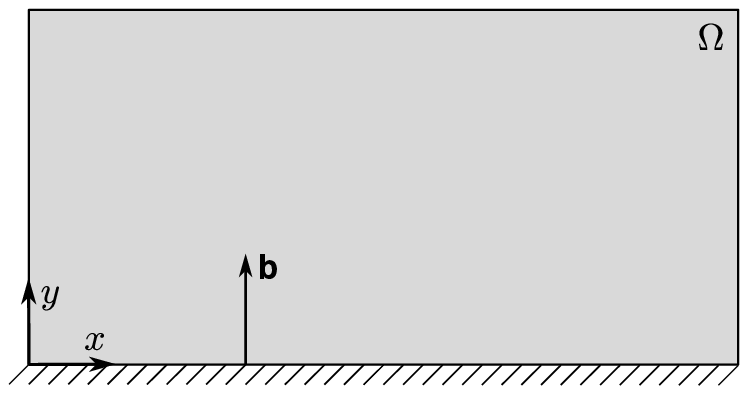}
         \caption{Sub-problems conditions.}
         \label{fig:short_cantilever_beam_geometry_boundary_conditions_sub}
     \end{subfigure}
     \caption{2D cantilever beam problem: geometry, loading, boundary conditions, and density field initial condition.}
     \label{fig:short_cantilever_beam_geometry_boundary_conditions}
\end{figure}

\begin{table}[h!]
\centering
\begin{tabular}{p{4.2cm} p{2.2cm} p{1.7cm} | p{3.4cm} p{1.8cm} p{1.0cm}}
 \hline
 \hline
 Parameter & Symbol & Value & Parameter & Symbol & Value \\
 \hline
  Dimension [$\mathrm{m}$] & $L_x \times L_y$ & $12 \times 6$ & 
    Filter radius [$\mathrm{m}$] & $\bar{r}$ & $1.25$ \\
 Thickness [$\mathrm{m}$] & $-$ & $1$ &
    Threshold parameters & $\beta_{\min}$ - $\beta_{\max}$ & $1$ - $32$ \\
 Average density [$\mathrm{Kg}$/$\mathrm{m}^3$] & $\bar{v}$ & $0.5$ &
    & $\beta_{d}$ & $100$ \\
 Quadrilateral elements & $N_x \times N_y$ & $240 \times 120$ &
    & $\eta$ & $0.5$ \\          
 \hline
 \hline
\end{tabular}
\caption{Complementary parameters for the 2D cantilever beam problem.}
\label{table:parameters_short_cantilever_beam_problem}
\end{table}

Figure \ref{fig:parameters_short_cantilever_beam_problem_initialaa} presents an optimized design obtained with the standard TO formulation \eqref{eq:standad_formulation_TO}, where the blue circle indicates the adopted filter size.
With a compliance of $J_D = 58.79\,\mathrm{Nm}$, the manufacturing of such structure, with this build direction, may not be appropriate for AM.
When the PUP formulation is employed, the designs of Figures \ref{fig:parameters_short_cantilever_beam_problem_qian_initial} and \ref{fig:parameters_short_cantilever_beam_problem_qian_II_initial}, referred to as structures I and II, are obtained.
These designs, having compliance of $J_D=86.38\,\mathrm{Nm}$ and $J_D=74.25\,\mathrm{Nm}$, were obtained by using filter radius $\bar{r} = 1.25\,\mathrm{m}$ and $\bar{r} = 0.50\,\mathrm{m}$, respectively. From the figures, we can clearly see the dependence of the final structure on $\bar{r}$.
The adopted PUP and grayness constraints upper limits are $\bar{P}_{45^\circ} = 0.5$ and $\bar{\epsilon} = 0.6$, with $\beta_d = 25$ starting from the $50$th iteration, {\color{blue}as proposed in \cite{qian2017undercut}.}
Figure \ref{fig:2D_classical_angle_versus_PUP_formulations_initial_a} shows their corresponding NPUP measures for different threshold angles.
Note that structure I exhibits smaller values of NPUP when compared to the standard structure, as expected, since it is based on the PUP formulation.
However, the reduction of the radius $\bar{r}$ may induce the formation 
of drippings, an issue generally present when adopting geometric AM constraints, as illustrated by structure II.
{\color{blue}
Although suppressing drippings be possible, it needs potentially onerous computations (see \cite{mezzadri2020second}, where the dripping effect is avoided by filters or constraints that require computing the second-order derivatives of the density function).}
As will be shown later on, the $\mathbb{AP}$ formulation proposed in this work
delivers improved designs obtained by using a small filter radius while avoiding the undesirable presence of drippings.

\begin{figure}[H]
    \centering
    \includegraphics[scale=0.36]{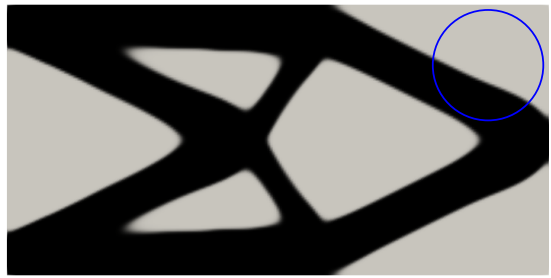}
    \caption{Standard reference structure.}
    \label{fig:parameters_short_cantilever_beam_problem_initialaa}
\end{figure}

\begin{figure}[H]
\centering
%\hspace{0.5cm}
\begin{minipage}{.45\textwidth}
\begin{subfigure}[t]{0.9\textwidth}
  \centering
  \includegraphics[scale=0.36]{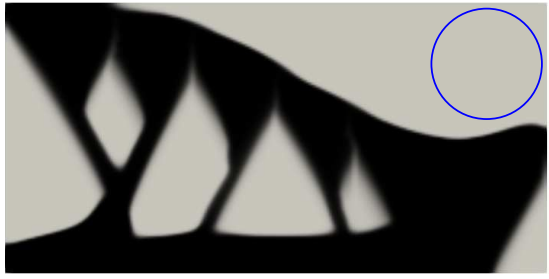}
  \captionof{figure}{Structure I. Suggested filter radius.}
  \label{fig:parameters_short_cantilever_beam_problem_qian_initial}
\end{subfigure}
\begin{subfigure}[t]{0.9\textwidth}
  \centering
  \includegraphics[scale=0.36]{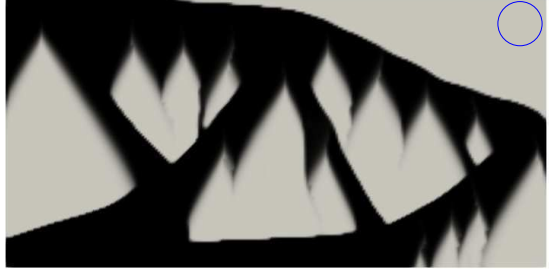}
  \captionof{figure}{Structure II. Reduced filter radius.}
  \label{fig:parameters_short_cantilever_beam_problem_qian_II_initial}
\end{subfigure}
\vspace{0.5cm}
\end{minipage}
\hspace{-0.5cm}
\begin{subfigure}{0.5\textwidth}
  \centering
  \includegraphics[scale=0.71]{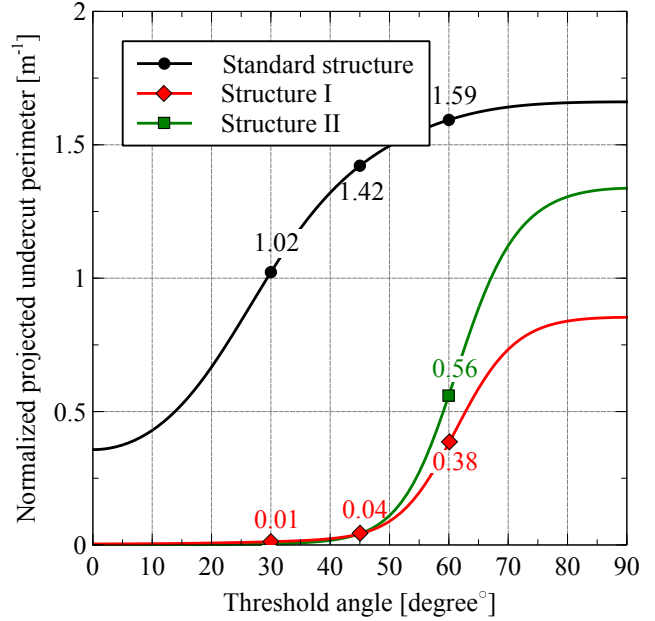}
  \captionof{figure}{NPUP measure with respect to the threshold angle for the standard structure and those obtained via PUP formulation.}
  \label{fig:2D_classical_angle_versus_PUP_formulations_initial_a}
\end{subfigure}
\caption{Structures with the suggested filter radius and a reduced filter radius version where the dripping effect is present.}
\label{fig:2D_classical_angle_versus_PUP_caption_initial}
\end{figure}

\subsubsection{Self-weight-based results}
\label{sec:2d_self_weight_based_results}

Let us present the results for the 2D cantilever beam problem by using the self-weight-based formulation described in Section \ref{sec:self_weight_based_formulation}.
Additionally to the parameters showed in Tables \ref{table:general_parameters_cantilever_beam_problem} and \ref{table:parameters_short_cantilever_beam_problem}, consider $w_0 = 0.10$.
The artificial body load is defined as $\bm{f}(\bar{\rho}\left( \rho \right)) = -g_{p} \, \bar{\rho} \, \bm{b}$, with the normalized gravity $g_{p}$ given by
\begin{equation*}
    g_{p} = \left.
     g \middle/ \left( \bar{v} \int_{\Omega} \mathrm{d}x \right) 
     \right. ,
\end{equation*}
allowing a balance between external and self-weight loads irrespective of the domain size.
Notice that, without this consideration $\bm{f}$ will be dominant (negligible) as compared to $\bm{t}$ for large (small) domains and consequently, the optimization processes will be controlled by the sub-problems (resp., the main problem).
In this case, a new $w_0$ weight adjustment becomes necessary.

As the first result, let us analyze the convergence of the  optimization process  with respect to the number of optimization steps.
Figure \ref{fig:complete_layers_x_compliance_02_with_figures_type2} illustrates the compliance, the total cost, and the grayness over the optimizer iterations when considering $40$ layers.
The process exhibits jumps as the $\beta$-threshold parameter increases from $\beta_{\min}$ to $\beta_{\max}$, 
aiming to produce densities close to a binary $\{0,1\}$ distribution.
The grayness curve is an indicator of such behavior, and in this case measures $4.11\%$ at the converged state.
Ideally, the grayness assumes the value of $0\%$ in the limit $\beta \to \infty$.
 At the end of the optimization process, the compliance is $J_D=78.91\,\mathrm{Nm}$. This value can be compared to the compliance of the structures of Figure \ref{fig:2D_classical_angle_versus_PUP_caption_initial}.
 The left part of the figure shows the corresponding designs obtained after {\color{blue}$150$(A) and $959$(B) (converged stage)} iterations, with the black region defining the {\color{blue}solid domain $\Omega^S$}.

\begin{figure}[H]
    \centering
    \includegraphics[scale=1.00]{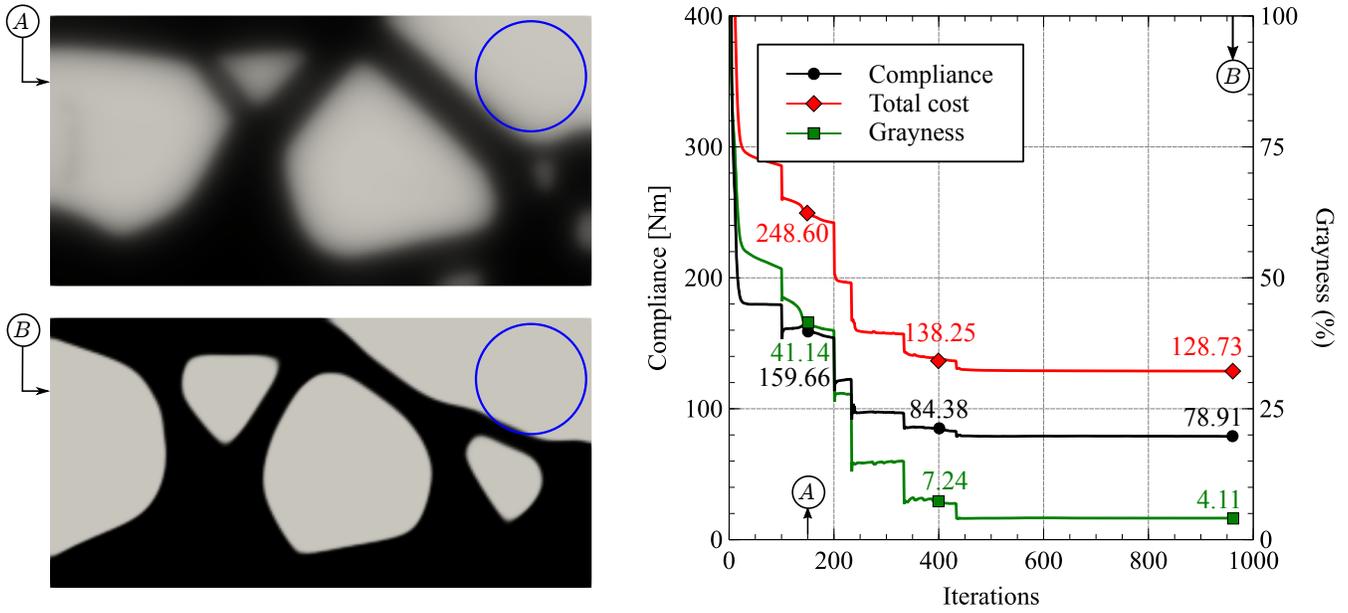}
    \caption{{\color{blue}The self-weight-based formulation for $40$ layers: Compliance, total cost, and grayness over the iterations. Density distribution details at the left.}}
    \label{fig:complete_layers_x_compliance_02_with_figures_type2}
\end{figure}

The number of layers plays a crucial role in the formulation as a discretization aspect which may affect the final design.
Figure \ref{fig:complete_layers_x_compliance_01_with_figures_type2} shows the compliance and NPUP values (for $30^\circ$, $45^\circ$ and $60^\circ$) as a function of the number of layers.
Details on the left show the designs obtained for {\color{blue}$5$} and $120$ layers.
These results also suggests a fast convergence of the compliance with respect to the number of layers.

\begin{figure}[H]
    \centering
    \includegraphics[scale=1.00]{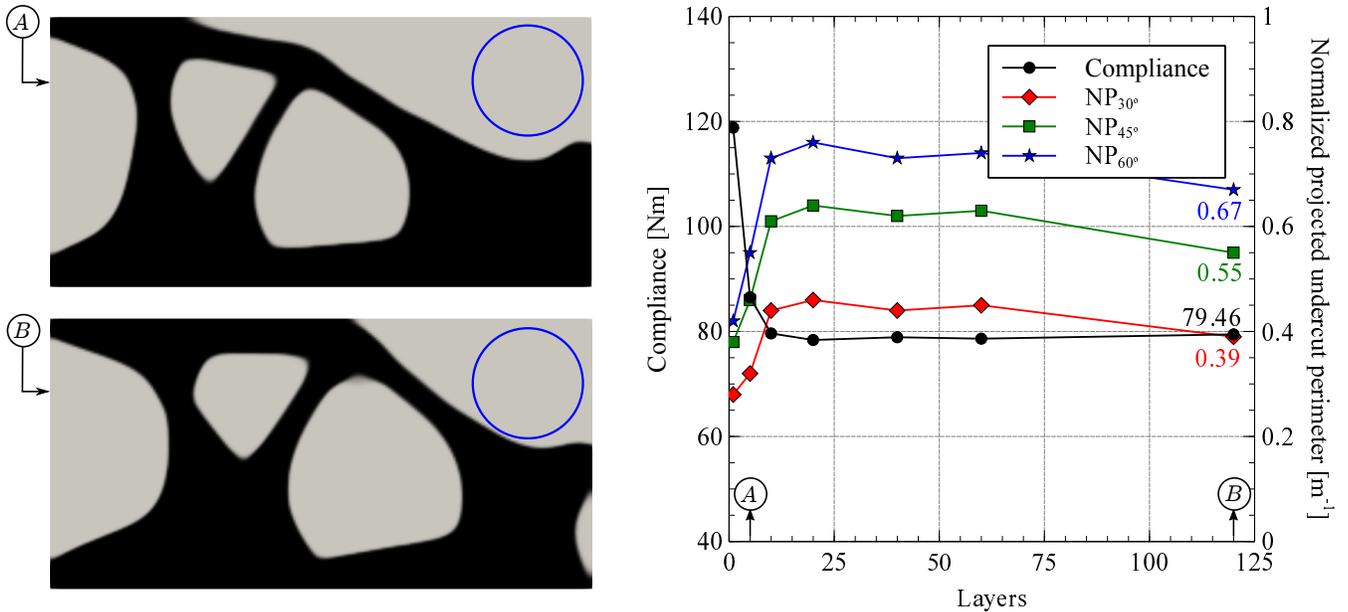}
    \caption{{\color{blue}The self-weight-based formulation: Compliance and NPUP values for different number of layers.}}
    \label{fig:complete_layers_x_compliance_01_with_figures_type2}
\end{figure}

The weight $w_0$ highly affects the final design, as illustrated in Figure \ref{fig:complete_weight_x_compliance_with_figures_type2}.
Note that a small value of $w_0$ leads to a high structure compliance and buildability reflected by the small NPUP values.
This behavior is reversed when increasing $w_0$.
These results show that the user has to choose a value of $w_0$ that balances these two antagonistic properties.
Figures at the left shows the designs obtained when using $w_0=0.2$ and $w_0=0.4$.

\begin{figure}[H]
    \centering
    \includegraphics[scale=1.0]{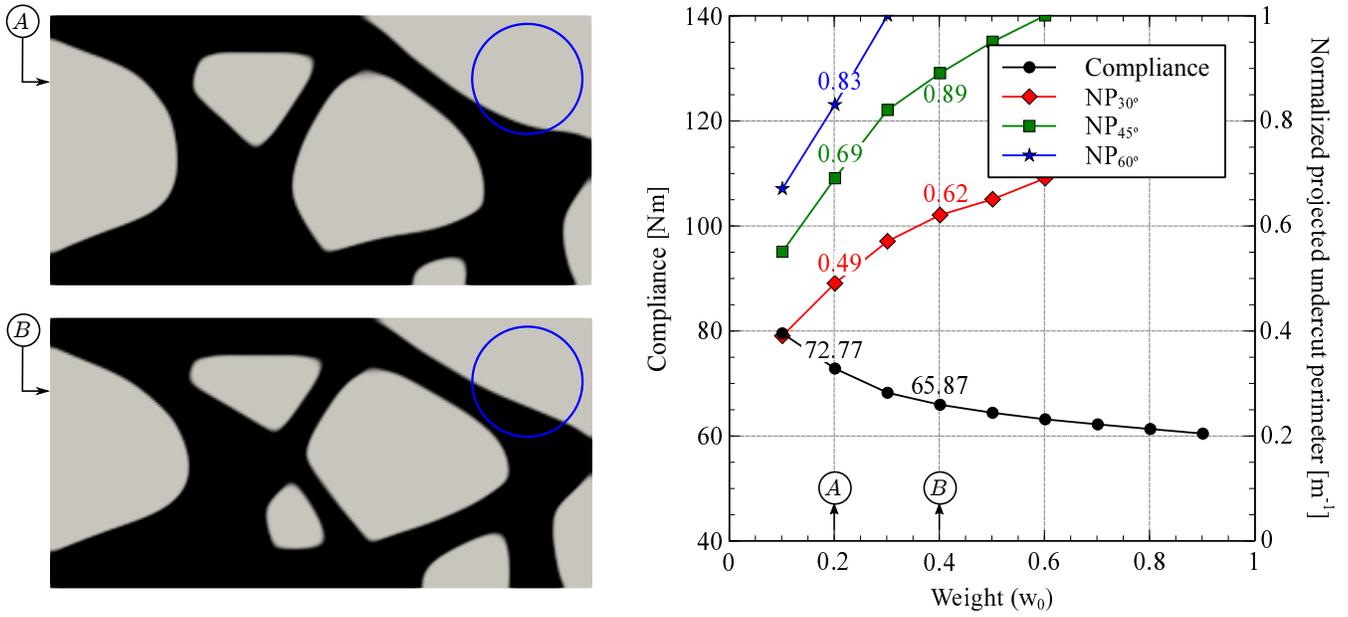}
    \caption{{\color{blue}The self-weight-based formulation for $120$ layers: Compliance and NPUP values for different weights.}}
    \label{fig:complete_weight_x_compliance_with_figures_type2}
\end{figure}

\begin{figure}[H]
    \centering
    \includegraphics[scale=1.00]{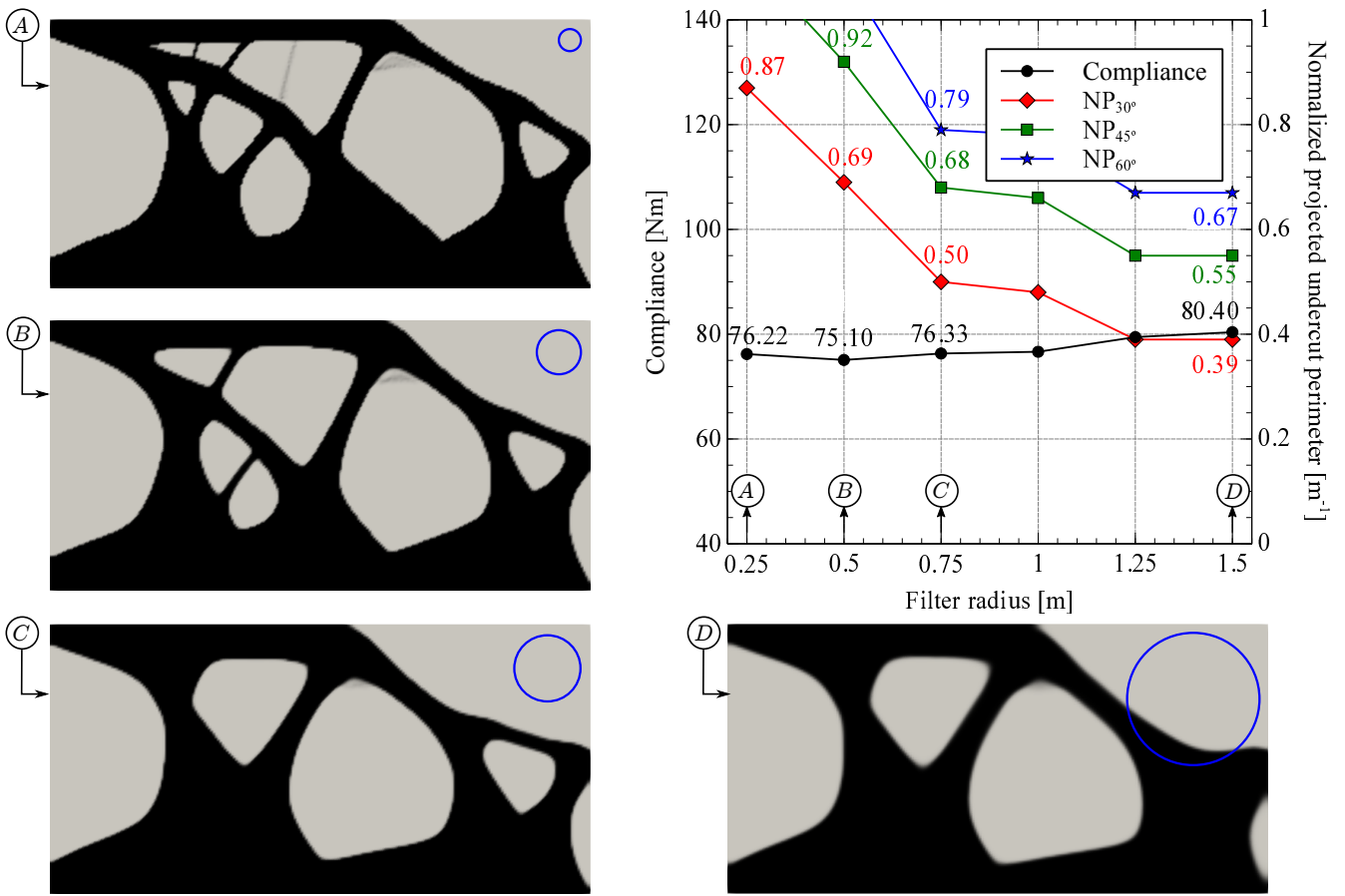}
    \caption{{\color{blue}The self-weight-based formulation for $120$ layers: Compliance and NPUP values for different filter radius.}}
    \label{fig:complete_radius_x_compliance_with_figures_type2}
\end{figure}

Another relevant variation is the filter radius.
As observed in the Figure \ref{fig:complete_radius_x_compliance_with_figures_type2}, designs with smallest NPUP values and high compliance are obtained with filter radius increasing.
According to the details for structures with $\bar{r}=0.25\,\mathrm{m}$, $0.5\,\mathrm{m}$, $0.75\,\mathrm{m}$ and $1.5\,\mathrm{m}$, larger filter radius acts to close holes, consequently avoiding overhang regions.
An alternative to decrease the NPUP values using small filter radius is to conveniently modify the weight $w_0$, as illustrated when discussed the results of Figure \ref{fig:complete_weight_x_compliance_with_figures_type2}. 
It is important to highlight that the dripping effect and sharp corners are not present in this formulation, even for the smallest filter radius.

\subsubsection{Thermal conductivity-based formulation}
\label{sec:res_thermal}

Let's focus now on the thermal conductivity-based formulation described in Section \ref{sec:thermal_conductivity_based_formulation}.
In addition to the parameters given in Tables \ref{table:general_parameters_cantilever_beam_problem} and \ref{table:parameters_short_cantilever_beam_problem}, consider {\color{blue}$w_0 = 0.25$} and $q_i = 1$ for $i=1,2,\ldots, l$.

Figure \ref{fig:heat_layers_x_compliance_02_with_figures_type2} shows the compliance, total cost, and grayness measure over the optimizer iteration for $40$ layers.
Similar to the discussion given concerning Figure \ref{fig:complete_layers_x_compliance_02_with_figures_type2}, we observe jumps, oscillations, and a non-decreasing behavior for the compliance curve due to the same factors previously discussed.
{\color{blue}
The final design has $J_D=84.27\,\mathrm{Nm}$ and $3.19\%$ of grayness.}
Details at the left show the corresponding designs after {\color{blue}$150$(A) and $585$(B)} (converged stage) iterations.

\begin{figure}[H]
    \centering
    \includegraphics[scale=1.00]{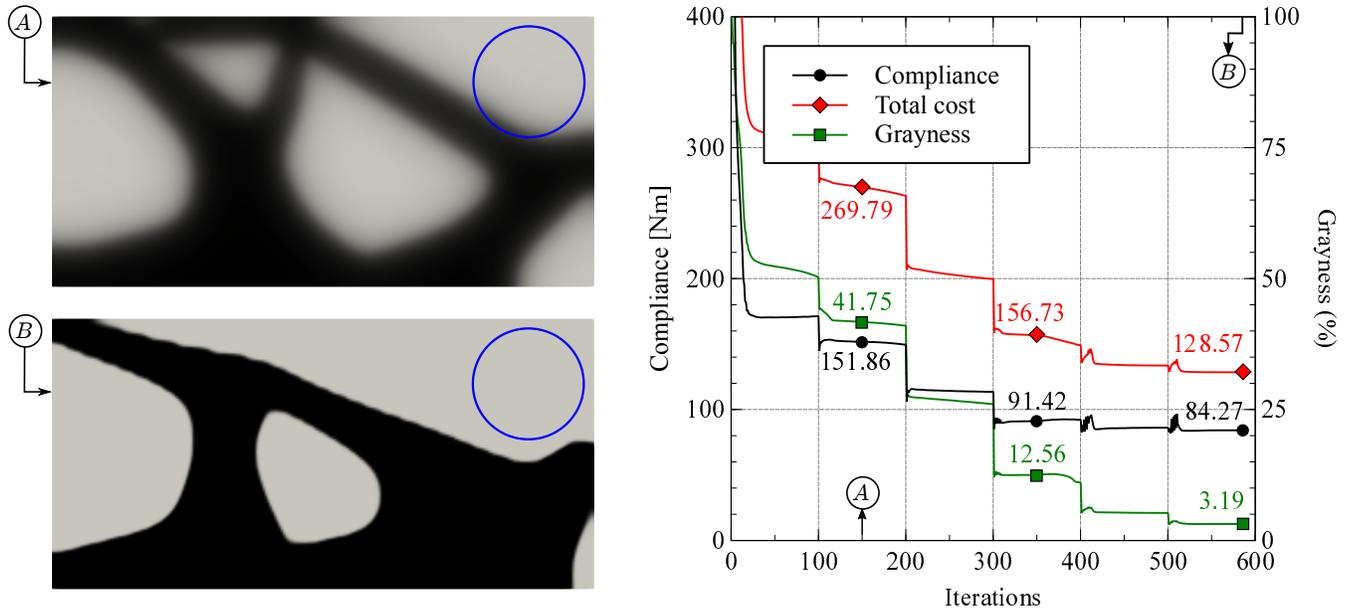}
    \caption{{\color{blue}Thermal conductivity-based formulation for $40$ layers: Compliance, total cost, and grayness over the iterations. Density distribution details at the left.}}
    \label{fig:heat_layers_x_compliance_02_with_figures_type2}
\end{figure}

The effect of the (number of) layers on the final design can be seen more clearly using this formulation, as illustrated in Figure \ref{fig:heat_layers_x_compliance_01_with_figures_type2}.
{\color{blue}
It is possible to observe boundary oscillations occurring between two consecutive layers due to the thermal loading applied on $\Gamma^u_i$.
As illustrated, they can be reduced with the layer number increasing}.
The compliance oscillates around $83\,\mathrm{Nm}$ for all cases, with the converged topology obtained for more than $40$ layers.

{\color{blue}Figures \ref{fig:heat_weight_x_compliance_with_figures_type2_120L} and \ref{fig:heat_radius_x_compliance_with_figures_type2_120L} show} the effects of the weight $w_0$ and the filter radius $\bar{r}$ on the compliance and NPUP values.
Similar to Section \ref{sec:2d_self_weight_based_results}, the increase of the weight leads to better structural responses at the cost of decreasing the manufacturability.
This behavior reverses with the filter radius increase.

\begin{figure}[H]
    \centering
    \includegraphics[scale=1.00]{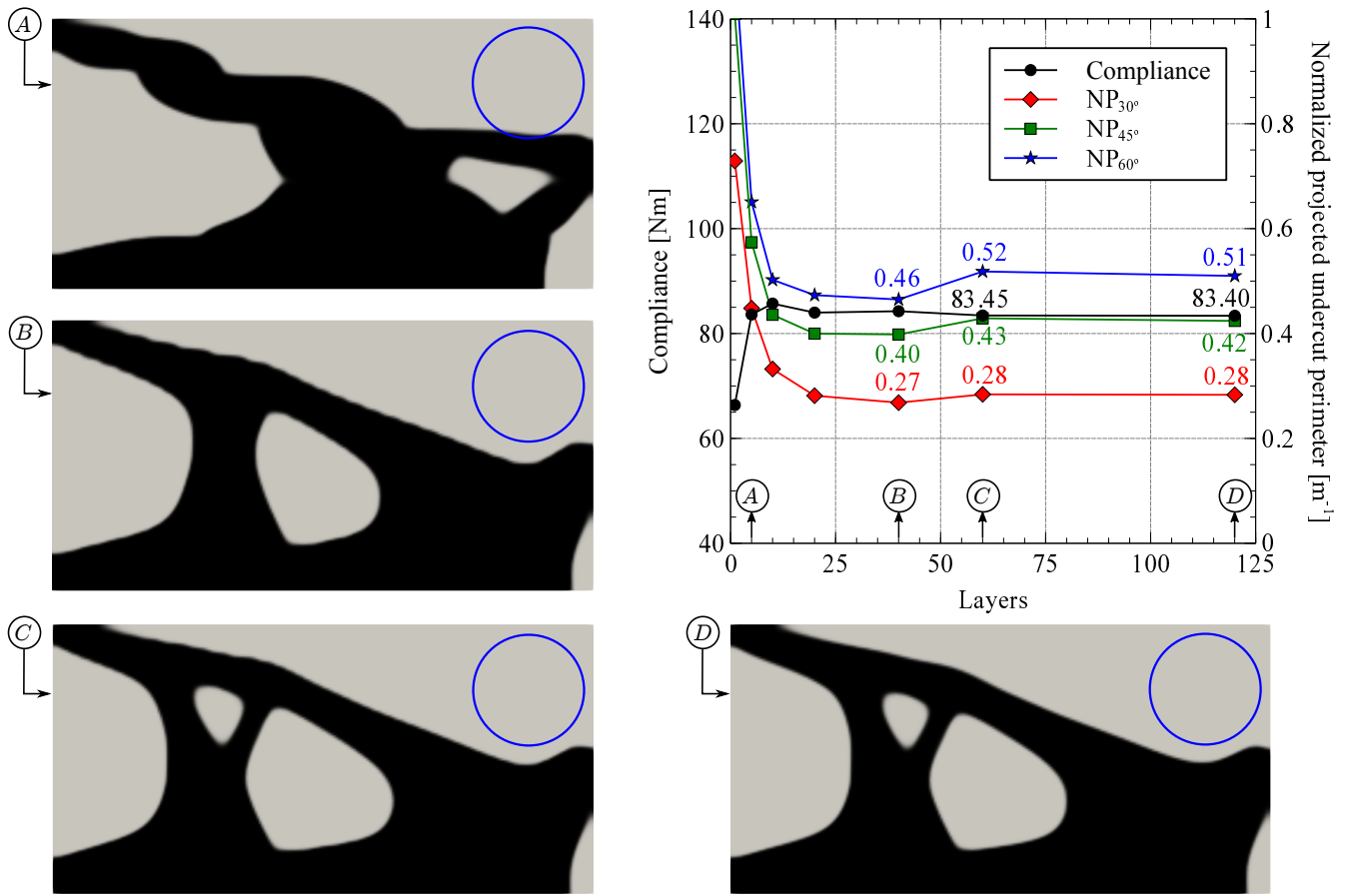}
    \caption{{\color{blue}Thermal conductivity-based formulation: Compliance and NPUP values for different layers.}}
    \label{fig:heat_layers_x_compliance_01_with_figures_type2}
\end{figure}

\begin{figure}[H]
    \centering
    \includegraphics[scale=1.0]{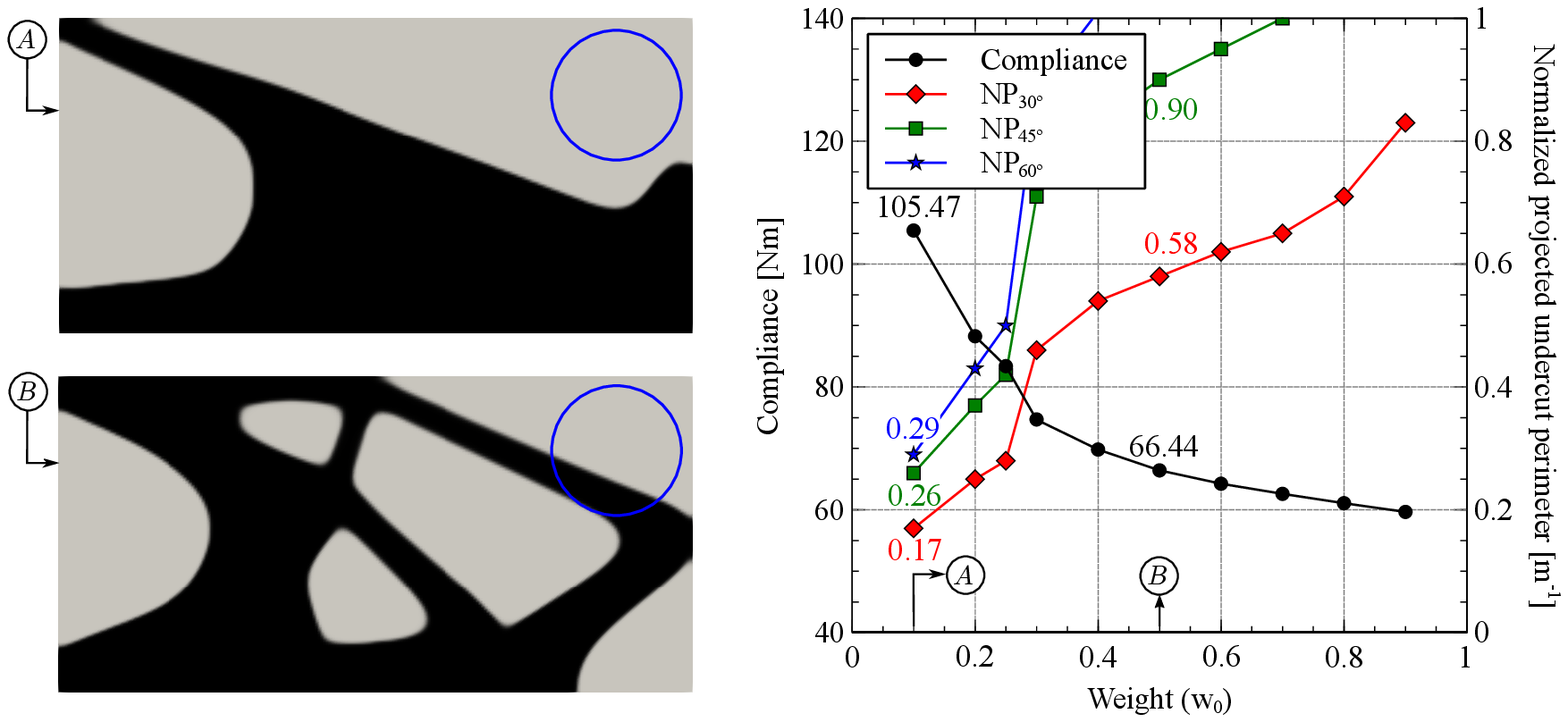}
    \caption{{\color{blue}Thermal conductivity-based formulation for $120$ layers: Compliance and NPUP values for different weights.}}
    \label{fig:heat_weight_x_compliance_with_figures_type2_120L}
\end{figure}

\begin{figure}[H]
    \centering
    \includegraphics[scale=1.00]{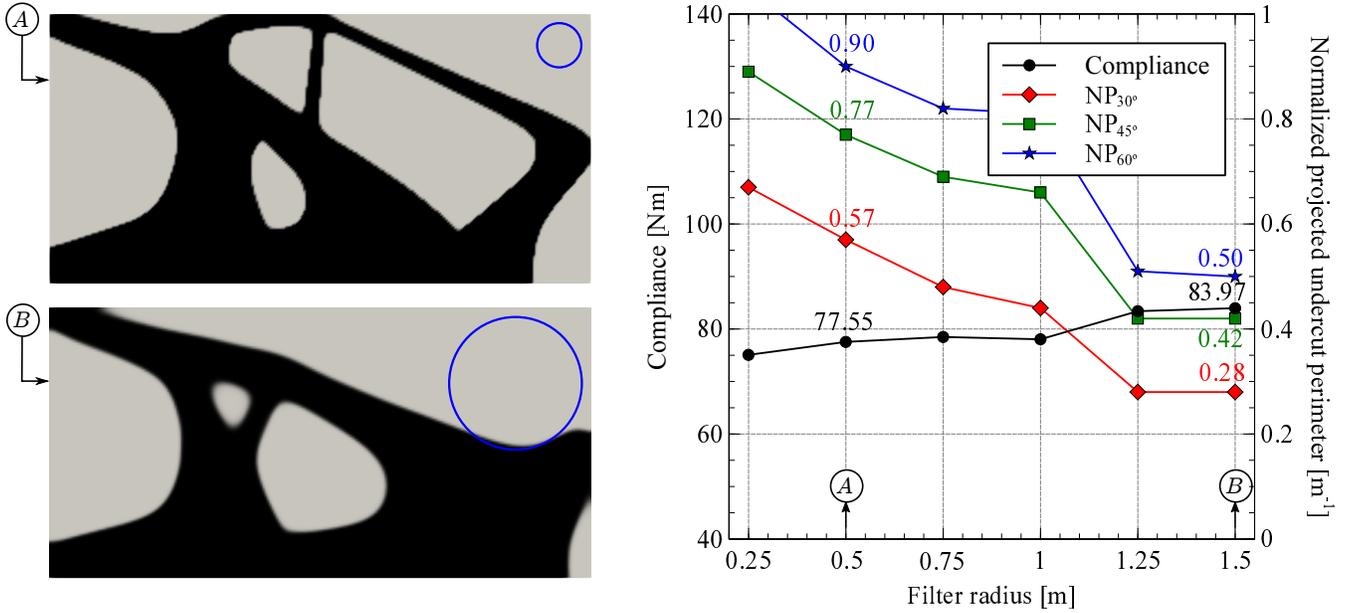}
    \caption{{\color{blue}Thermal conductivity-based formulation for $120$ layers: Compliance and NPUP values for different filter radius.}}
    \label{fig:heat_radius_x_compliance_with_figures_type2_120L}
\end{figure}

{\color{blue}
Finally, in order to check the sensitivity of the optimized design concerning the initial density field, we performed the simulation D of Figure \ref{fig:heat_layers_x_compliance_01_with_figures_type2}, changing the initial condition as indicated in Figure \ref{fig:init_densig_field_casea}.
The result, shown in Figure \ref{fig:init_densig_field_caseb} indicates a similar topology of the one illustrated in Figure
\ref{fig:heat_layers_x_compliance_01_with_figures_type2}(D).}

\begin{figure}[H]
     \centering
    \begin{subfigure}[b]{0.45\textwidth}
         \centering
         \includegraphics[scale=0.36]{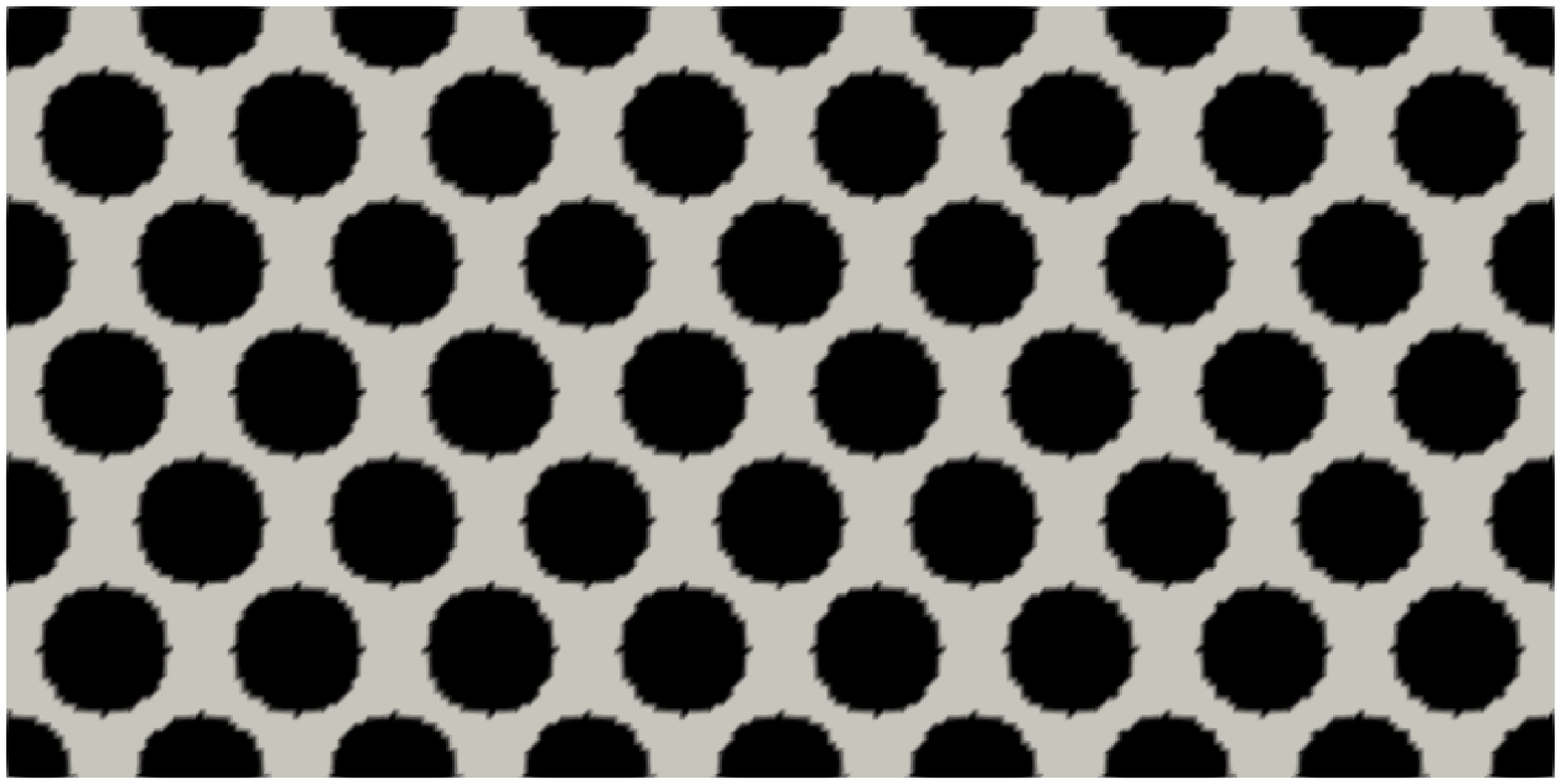}
         \caption{Initial condition for the density field.}
         \label{fig:init_densig_field_casea}
     \end{subfigure}
     \hspace{5pt}
     \begin{subfigure}[b]{0.45\textwidth}
         \centering
         \includegraphics[scale=0.36]{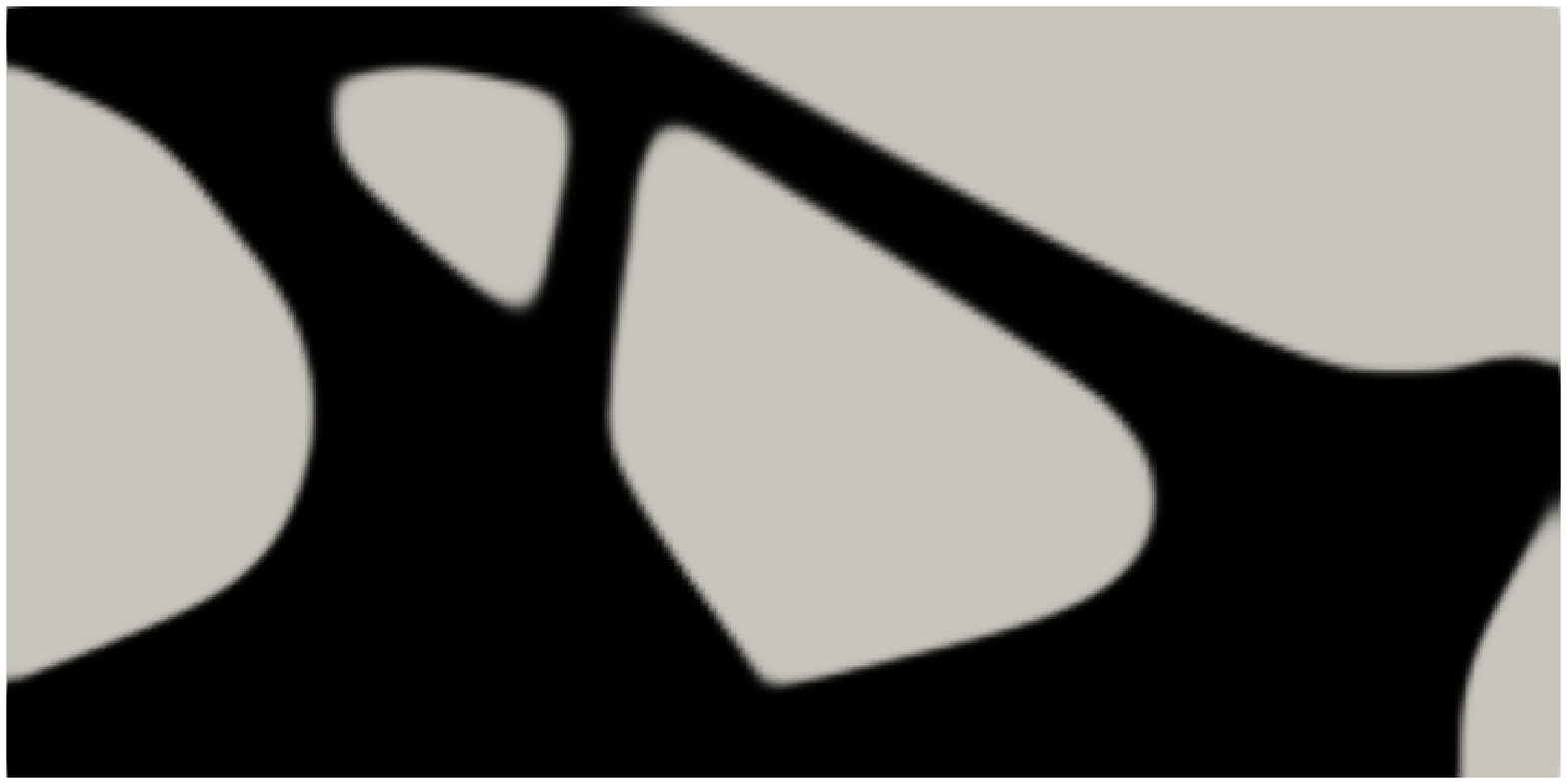}
         \caption{Resulting design after the optimization.}
         \label{fig:init_densig_field_caseb}
     \end{subfigure}
     \caption{{\color{blue}Thermal conductivity-based formulation for $120$ layers: Alternative initial condition for the density field.}}
     \label{fig:heat_layers_x_compliance_01_with_figures_type2_initial_condition}
\end{figure}

\subsubsection{Comparison and performance issues}
\label{sec:comparison}

As observed in the previous sections (see Figures \ref{fig:complete_radius_x_compliance_with_figures_type2} and \ref{fig:heat_radius_x_compliance_with_figures_type2_120L}), the $\mathbb{AP}$ formulation enables the use of small filter radius preventing the appearance of the undesirable dripping effect and sharp corners.
These resulting designs are alternative solutions for reducing overhang regions, or equivalently, increasing manufacturability by AM means.
For a fair comparison, Figure \ref{fig:2D_classical_angle_versus_PUP_formulations_III} shows the NPUP values considering different threshold angles for the standard reference structure and those obtained via $\mathbb{AP}$ formulation with $\bar{r}=0.50\,\mathrm{m}$.
There is a substantial reduction in NPUP values when compared to the standard structure.
{\color{blue}
Although they are higher than those provided when using PUP formulation for $\bar{\alpha} < 60^\circ$, the latter may leads to the formation of drippings, as illustrated in Figure \ref{fig:2D_classical_angle_versus_PUP_formulations_initial_a}
(see discussion on the second paragraph of Section \ref{sec:2d_cantilever_beam}).}

It is important to highlight {\color{blue}that, instead of working} with Heaviside functions in a fixed mesh $\mathcal{M}$ of $\Omega$ to mimic the build process, our implementation considers a particular mesh $\mathcal{M}_i$ for each domain $\Omega_i \subset \Omega$. 
This enables to save computational time when solving the sub-problems since they are solved for each $\Omega_i$. 
This comes at the price of having to map the density degrees of freedom between different meshes, which is done easily and at a low cost in the FEniCS platform.
Figure \ref{fig:2D_layers_x_time_II} illustrates the average of CPU time expended in each optimizer iteration.
The thermal conductivity-based formulation is evidently more attractive for a purely computational point-of-view than the self-weight-based version.
This is due to the reduced size of the linear system to be solved when compared to the one associated with the linear elasticity balance equation.

\begin{minipage}{.44\textwidth}
\hspace{-0.5cm}
\begin{figure}[H]
    \centering
    \includegraphics[scale=0.71]{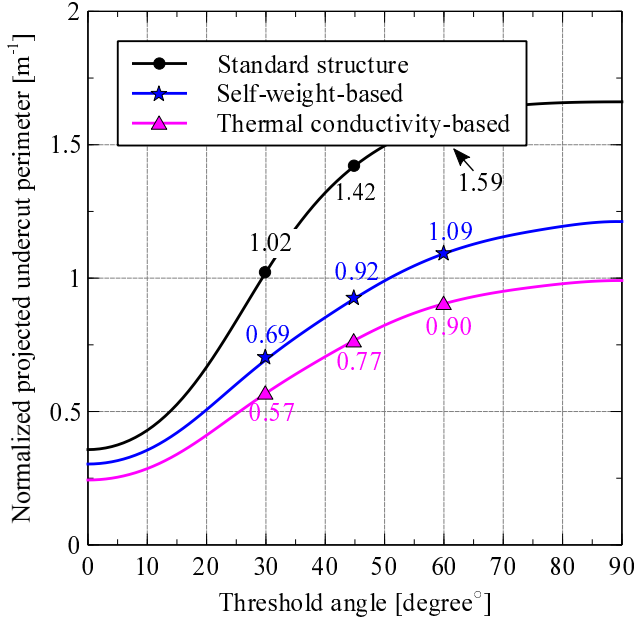}
    \caption{{\color{blue}NPUP measure with respect to the threshold angle for the standard structure and those obtained via $\mathbb{AP}$ formulation.}}
    \label{fig:2D_classical_angle_versus_PUP_formulations_III}
\end{figure}
\end{minipage}
\hspace{1cm}
\begin{minipage}{.44\textwidth}
\begin{figure}[H]
    \centering
    \vspace{-0.4cm}
    \includegraphics[scale=0.71]{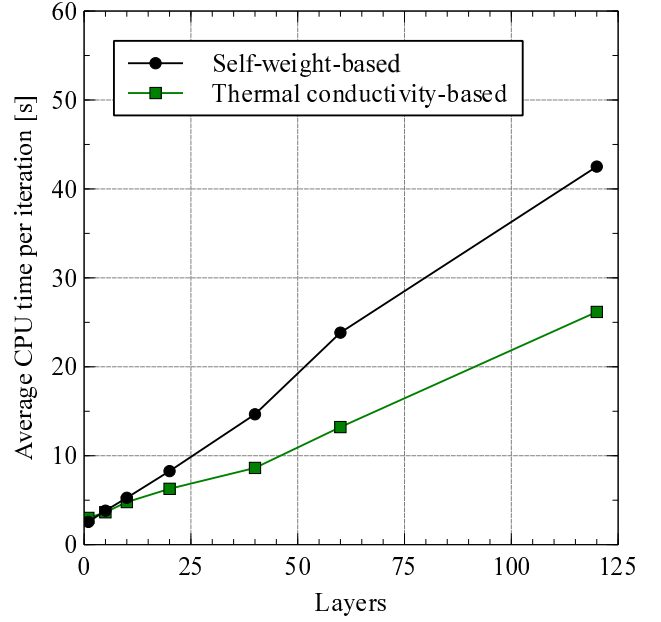}
    \caption{Number of layers versus average CPU time per iteration.}
    \label{fig:2D_layers_x_time_II}
\end{figure}
\end{minipage}

\subsection{2D MBB beam}
\label{sec:2d_mbb}

Consider the symmetric MBB beam under the hypothesis of the plane-strain state illustrated by Figure \ref{fig:mbb_beam_geometry_boundary_conditions}, subject to the parameters given in Table \ref{table:parameters_mbb_beam_problem}.
\citet{amir2018topology} proposed a similar problem using self-weight-based loading for the sub-problems as a way to reduce overhang surfaces; however, the corresponding formulations are not equivalent.
The purpose of this section is to apply our self-weight-based and thermal conductivity-based formulations to verify the formulation's generality when applied to a different set of boundary conditions.

\begin{figure}[H]
     \centering
    \begin{subfigure}[b]{0.45\textwidth}
         \centering
         \includegraphics[scale=1.00]{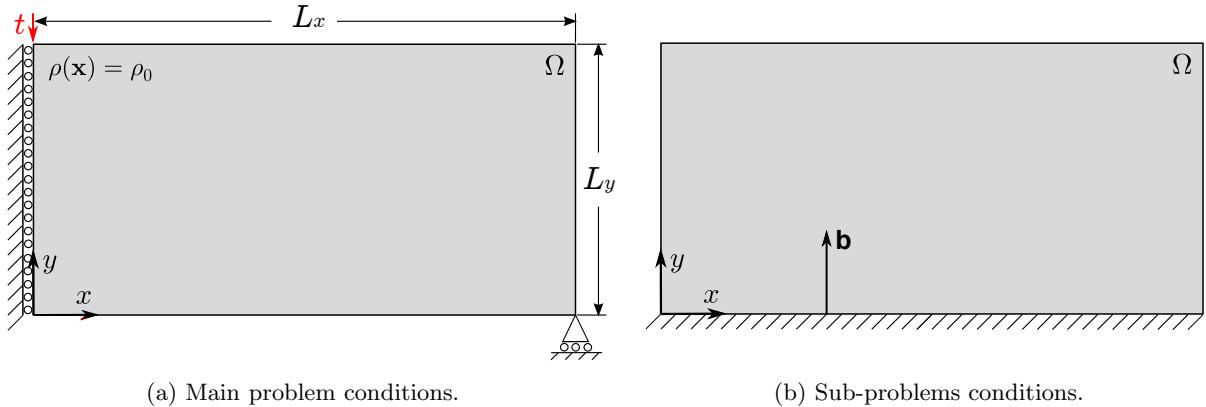}
         \caption{Main problem conditions.}
         \label{fig:mbb_beam_geometry_boundary_conditions_main}
     \end{subfigure}
     \hspace{5pt}
     \begin{subfigure}[b]{0.45\textwidth}
         \centering
         \includegraphics[scale=1.00]{4B_18B.eps}
         \vspace{0.4cm}
         \caption{Sub-problems conditions.}
         \label{fig:mbb_beam_geometry_boundary_conditions_sub}
     \end{subfigure}
     \caption{2D MBB beam problem: geometry, loading, boundary conditions, and density field initial condition.}
     \label{fig:mbb_beam_geometry_boundary_conditions}
\end{figure}

\begin{table}[h!]
\centering
\begin{tabular}{p{4.2cm} p{2.2cm} p{1.7cm} | p{3.4cm} p{1.8cm} p{1.0cm}}
 \hline
 \hline
 Parameter & Symbol & Value & Parameter & Symbol & Value \\
 \hline
  Dimension [m] & $L_x \times L_y$ & $160 \times 80$ & 
    Filter radius [m] & $\bar{r}$ & $2.40$ \\
 Thickness [m] & $-$ & $1$ &
  Threshold parameters & $\beta_{\min}$ - $\beta_{\max}$ & $1$ - $4$ \\
 Average density [Kg/m$^3$] & $\bar{v}$ & $0.6$ &
    & $\beta_{d}$ & $100$ \\
 Quadrilateral elements & $N_x \times N_y$ & $320 \times 160$ &
    & $\eta$ & $0.5$ \\
 \hline
 \hline
\end{tabular}
\caption{Complementary parameters for the 2D MBB beam problem.}
\label{table:parameters_mbb_beam_problem}
\end{table}

{\color{blue}
Figure \ref{fig:mbb_beam_problem} illustrates the optimized structure obtained via the standard TO formulation \eqref{eq:standad_formulation_TO}, whereas Figure \ref{fig:2D_mbb_classical_angle_versus_PUP_caption_initial} contemplates the results by using the self-weight (\ref{fig:mbb_beam_problem_selfweight}), and thermal conductivity-based (\ref{fig:mbb_beam_problem_heat}) formulations with $w_0=0.10$ and $w_0=0.99$, respectively.
Both strategies presented competitive compliance values and overhang surface reduction, as indicated by the Figure \ref{fig:2D_mbb_classical_angle_versus_PUP_formulations_initial_a}.}

\begin{figure}[H]
    \centering
    \includegraphics[scale=0.36]{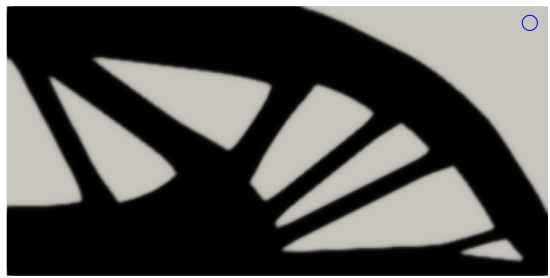}
    \caption{Standard reference 2D MBB beam design. Compliance $J_D=65.58\,\mathrm{Nm}$.}
    \label{fig:mbb_beam_problem}
\end{figure}

\begin{figure}[H]
\centering
%\hspace{0.5cm}
\begin{minipage}{.45\textwidth}
\begin{subfigure}[t]{0.9\textwidth}
  \centering
  \includegraphics[scale=0.36]{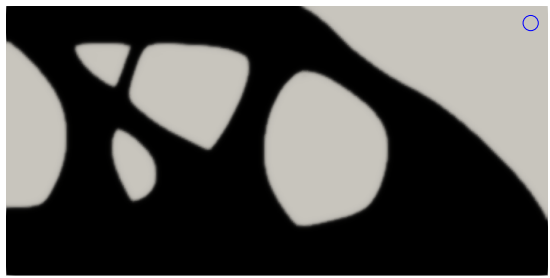}
  \captionof{figure}{Self-weight-based formulation with $w_0=0.10$. Compliance of $J_D=73.84\,\mathrm{Nm}$}
  \label{fig:mbb_beam_problem_selfweight}
\end{subfigure}
\begin{subfigure}[t]{0.9\textwidth}
  \centering
  \includegraphics[scale=0.36]{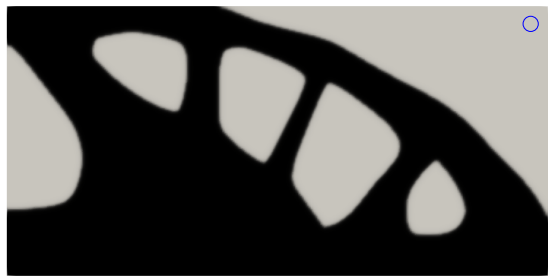}
  \captionof{figure}{Thermal conductivity-based formulation with $w_0=0.99$. Compliance $J_D=74.30\,\mathrm{Nm}$.}
  \label{fig:mbb_beam_problem_heat}
\end{subfigure}
\vspace{0.5cm}
\end{minipage}
\hspace{-0.5cm}
\begin{subfigure}{0.5\textwidth}
  \centering
  \includegraphics[scale=0.71]{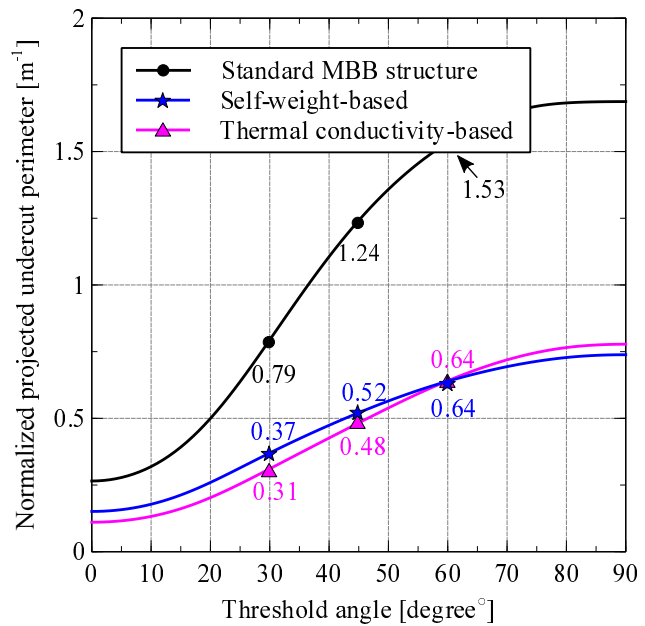}
  \captionof{figure}{NPUP measure with respect to the threshold angle for the standard MBB structure and those obtained via self-weight-based and thermal conductivity-based formulations.}
  \label{fig:2D_mbb_classical_angle_versus_PUP_formulations_initial_a}
\end{subfigure}
\caption{{\color{blue}Structures obtained using our proposed formulations for $160$ layers.}}
\label{fig:2D_mbb_classical_angle_versus_PUP_caption_initial}
\end{figure}

\subsection{3D beam}
\label{sec:3d_cantilever_beam}

We now consider a 3D version of the previous cantilever beam problem illustrated in Figure \ref{fig:3d_beam_geometry_boundary_conditions}.
As shown in Table \ref{table:3d_parameters_short_cantilever_beam_problem}, there are minor differences in the parameters that aim to deal with the coarse mesh and the computational cost of the 3D simulations.
The adopted target density generates the same volume of solid material as in the 2D case, enabling a fairer comparison.
Although possible, the symmetry around $y=3$ was not used in this study.
For illustrative purposes, the 3D results are obtained using contour, extract surface, and clip filters in the ParaView software \cite{ahrens2005paraview}.

\begin{figure}[H]
     \centering
    \begin{subfigure}[b]{0.45\textwidth}
         \centering
         \includegraphics[scale=1.00]{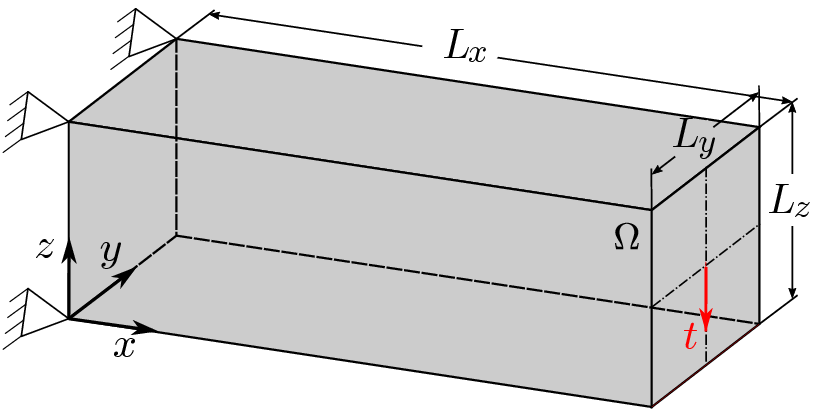}
         \caption{Main problem conditions.}
         \label{fig:3d_beam_geometry_boundary_conditions_main}
     \end{subfigure}
     \hspace{5pt}
     \begin{subfigure}[b]{0.45\textwidth}
         \centering
         \includegraphics[scale=1.00]{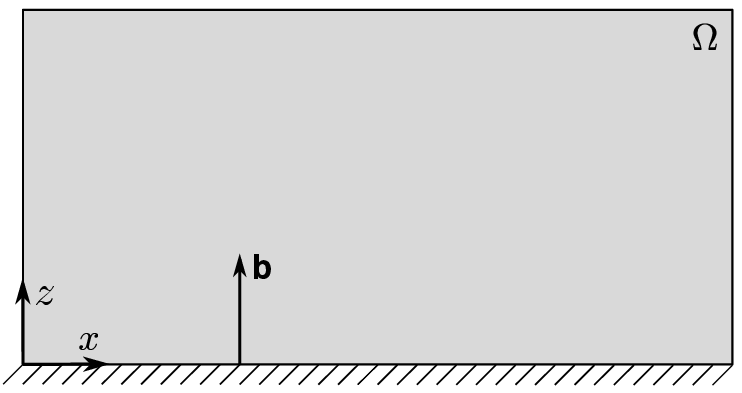}
         \caption{Sub-problems conditions.}
         \label{fig:s3d_beam_geometry_boundary_conditions_sub}
     \end{subfigure}
     \caption{3D beam problem: geometry, loading, boundary conditions, and density field initial condition.}
     \label{fig:3d_beam_geometry_boundary_conditions}
\end{figure}

\begin{table}[h!]
\centering
\begin{tabular}{p{4.2cm} p{2.2cm} p{1.7cm} | p{3.4cm} p{1.8cm} p{1.0cm}}
 \hline
 \hline
 Parameter & Symbol & Value & Parameter & Symbol & Value \\
 \hline
  Dimension [m] & $L_x \times L_y \times L_z$ & $12 \times 6 \times 6$ & 
    Threshold parameters & $\beta_{\min}$ - $\beta_{\max}$ & $1$ - $4$ \\
 Average density [Kg/m$^3$] & $\bar{v}$ & $0.08\overline{33}$ &
    & $\beta_{d}$ & $50$ \\
 Hexahedron elements & $N_x \times N_y \times N_z$ & $60 \times 30 \times 30$ &
    & $\eta$ & $0.5$ \\         
 \hline
 \hline
\end{tabular}
\caption{{\color{blue}Complementary parameters for the 3D beam problem.}}
\label{table:3d_parameters_short_cantilever_beam_problem}
\end{table}

Figure \ref{fig:Classical_perspective} illustrates the optimized design obtained via standard TO formulation \eqref{eq:standad_formulation_TO}.
Qualitatively similar to the corresponding 2D case (see Figure \ref{fig:parameters_short_cantilever_beam_problem_initialaa}), it is obtained by using {\color{blue} $\bar{r}=0.50\,\mathrm{m}$ resulting in $J_D=96.45\,\mathrm{Nm}$.}

\begin{figure}[H]
\centering
\includegraphics[scale=0.85]{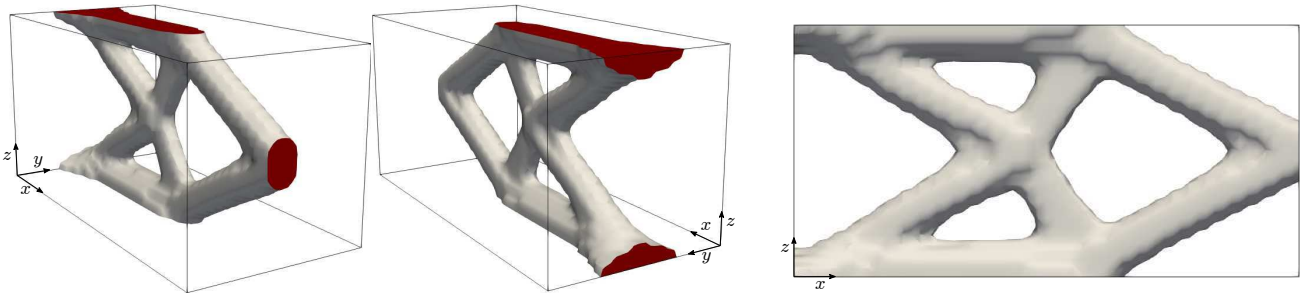}
\caption{{\color{blue}Standard reference structure. Design with $J_D=96.45\,\mathrm{Nm}$.}}
\label{fig:Classical_perspective}
\end{figure}

\begin{table}[h!]
\centering
\begin{tabular}{p{4cm}p{3cm}p{1.5cm}p{2cm}p{2cm}p{2cm}}
 \hline
 \hline
 Formulation & Filter radius [m] & $J_D$ [Nm] & $NP_{30^\circ}$ [m$^{-1}$] & $NP_{45^\circ}$ [m$^{-1}$] & $NP_{60^\circ}$ [m$^{-1}$] \\
 \hline
 \multirow{2}{*}{Standard} & $0.50$ & $96.45$ & $0.08$ & $0.16$ & $0.24$ \\
                            & $1.25$ & $235.13$ & $0.05$ & $0.11$ & $0.17$ \\
 \hline
 \multirow{2}{*}{PUP} & $0.50$ & $129.12$ & $0.01$ & $0.03$ & $0.11$ \\
                       & $1.25$ & $251.66$ & $0.01$ & $0.03$ & $0.09$ \\
 \hline
 \multirow{2}{*}{Self-weight} & $0.50$ & $126.88$  & $0.04$ & $0.07$ & $0.11$ \\
                              & $1.25$ & $303.03$ & $0.03$ & $0.06$ & $0.09$ \\
 \hline
 \multirow{2}{*}{Thermal conductivity} & $0.50$ & $127.83$ & $0.03$ & $0.05$ & $0.08$ \\
                                       & $1.25$ & $294.84$ & $0.02$ & $0.04$ & $0.08$ \\
 \hline
 \hline
\end{tabular}
\caption{{\color{blue}Quantitative values for the 3D cases.}}
\label{table:quantitative_values_3d}
\end{table}

Figure \ref{fig:perspective_I} shows designs obtained via PUP formulation (\ref{fig:Classical_perspective_45degree}), self-weight (\ref{fig:Complete_perspective_15Layers_smallradius}), and thermal conductivity-based (\ref{fig:Heat_perspective_15Layers_smallradius}) formulations {\color{blue} with $\bar{r}=0.50\,\mathrm{m}$}.
Quantitative values are given in Table \ref{table:quantitative_values_3d}.
The PUP formulation is employed by using the constraint upper limits of $P_{\bar{\alpha}}=2.0$ and $\bar{\epsilon}=0.5$.
It differs from the others in the topology (in the sense that there is material in a $V$ format along the $y$-axis) and due to the apparent dripping effect.
Showing {\color{blue}$J_D=129.12\,\mathrm{Nm}$} and relative small NPUP values, such design is unfeasible for AM point-of-view as the corresponding 2D case. 
The designs for the self-weight and thermal conductivity-based formulations are obtained by considering {\color{blue}$30$} layers and the weights of $w_0=0.10$ and {\color{blue}$w_0=0.25$} (similar to the corresponding 2D cases), presenting compliances of {\color{blue}$J_D=126.88\,\mathrm{Nm}$ and $J_D=127.83\,\mathrm{Nm}$}, respectively.
Designs obtained via $\mathbb{AP}$ formulation shows larger NPUP values but do not exhibit drippings.

\begin{figure}[H]
    \centering
    \begin{subfigure}[t]{1.0\textwidth}
      \centering
      \includegraphics[scale=0.85]{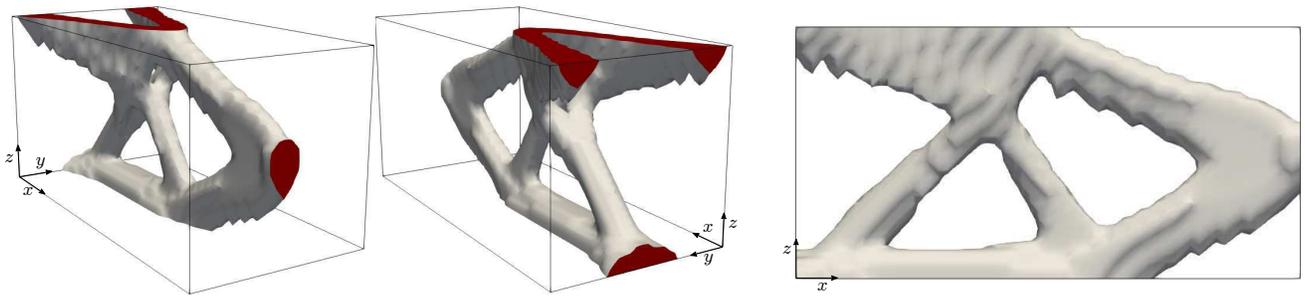}
      \captionof{figure}{PUP formulation. Design with $J_D=129.12\,\mathrm{Nm}$.}
      \label{fig:Classical_perspective_45degree}
    \end{subfigure}
    \begin{subfigure}[t]{1.0\textwidth}
      \centering
      \includegraphics[scale=0.85]{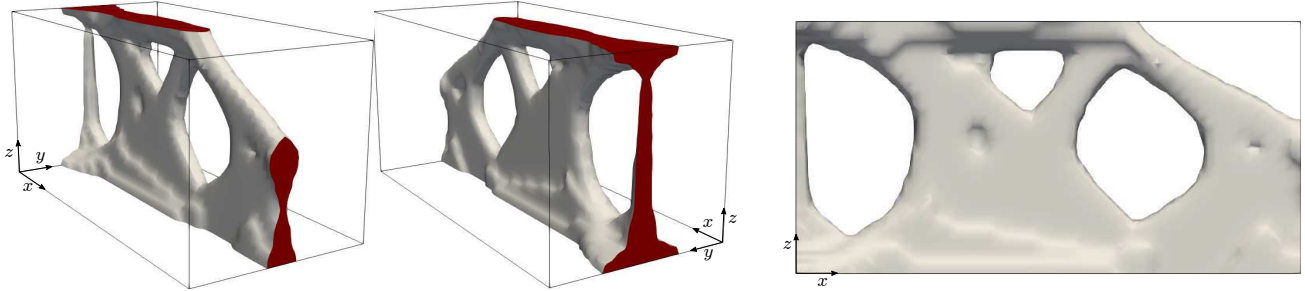}
      \captionof{figure}{Self-weight-based formulation for $30$ layers and $w_0=0.10$. Design with $J_D=126.88\,\mathrm{Nm}$.}
      \label{fig:Complete_perspective_15Layers_smallradius}
    \end{subfigure}
    \begin{subfigure}[t]{1.0\textwidth}
      \centering
      \includegraphics[scale=0.85]{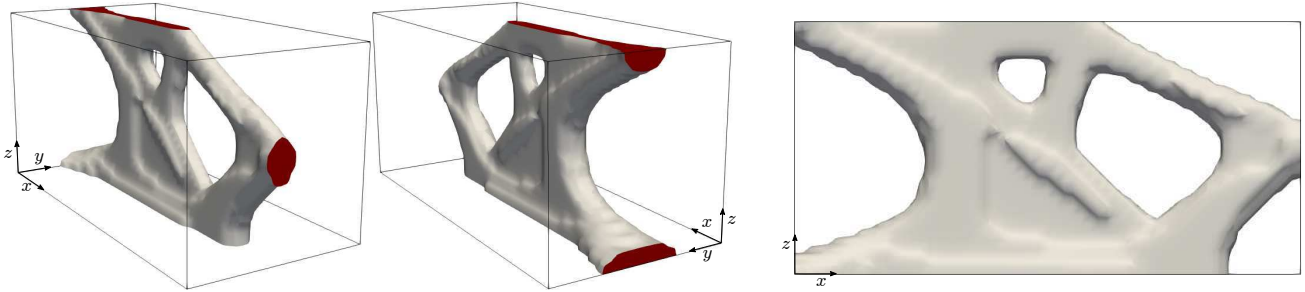}
      \captionof{figure}{Thermal conductivity-based formulation for 30 layers and $w_0=0.25$. Design with $J_D=127.83\,\mathrm{Nm}$.}
      \label{fig:Heat_perspective_15Layers_smallradius}
    \end{subfigure}
    \caption{{\color{blue}3D structures obtained via PUP and $\mathbb{AP}$ formulations.}}
    \label{fig:perspective_I}
\end{figure}

\section{Conclusion}
\label{sec:conclusion}

This study proposed in the form of problem ($\mathbb{AP}$) a general formulation of TO aimed at AM that accounts for the AM layer-by-layer building process by including a cost of partially built structures in the functional to be minimized. The formulation was specialized by taking it as being either the mechanical compliance, due to self-weight of the partially built structures, or the thermal compliance,
due to a heat source (the laser) applied on the top layer of the partially built structures. These costs lead to a limitation of the overhang relative to the build plate, and the numerical results show that the self-weight and the thermal conductivity-based formulations improve the additive manufacturability by reducing the normalized projected undercut perimeter of the reference structure at the cost of a small compliance increase. Both formulations worked well with a small filter radius without the appearance of dripping effects and sharp corners often encountered in geometric-based constraints. Optimizations, in both 2D and 3D, adopting different dimensions and boundary conditions show the formulation's generality.

The performed implementation using well-consolidated packages and the strategy of solving reduced systems (for each partial structure) allowed to carry out simulations in acceptable times. Parallelization can be added to improve this brand.

\section{Acknowledgment}

The authors  thankfully acknowledge financial support from the S\~ao Paulo Research Foundation (FAPESP) under grants CEPID-CeMEAI 2013/07375-0 and 2020/14288-0, and from the National Council for Scientific and Technological Development - CNPq under grants 304192/2019-8 and 304074/2020-9.

\bibliographystyle{plainnat}
\bibliography{bibliography}

\appendix

\section{Proof of Theorem 1}
\label{sec:app_theorem1}
\begin{proof} Let $(\hat{\rho}(\rho_{n_{k}}))$ be an arbitrary subsequence. From the coercivity of the bi-linear form, Cauchy-Schwartz and the fact that $\rho \leq 1$ a.e.\ we deduce that $||\hat{\rho}(\rho_{n_{k}})||_{1,\Omega} \leq |\Omega|/\min\{1,r^2\}$ for all $n_{k}$. We can thus extract a subsequence such that the filtered designs converge weakly in $H^{1}(\Omega)$, and thus also in $L^{2}(\Omega)$, to a solution $\hat{\rho} = \hat{\rho}(\rho)$ to \eqref{eq:variational_formulation_FEM}. Appealing to the Rellich-Kontrachev theorem \cite[Theorem 9.16]{brezis2011functional} we may extract a further subsequence converging strongly in $L^{2}(\Omega)$ to some $\overline{\rho}$. Since a strong limit is a weak limit and the latter is unique it follows that $\overline{\rho} = \hat{\rho}(\rho)$, i.e.\ we have obtained a subsequence of designs such that the filtered designs converge strongly in $L^{2}(\Omega)$ to $\hat{\rho}(\rho)$. According to Theorem 4.9 in \cite{brezis2011functional} we may then extract a further subsequence converging point-wise a.e.\ to $\hat{\rho}(\rho)$. 
Since \eqref{eq:variational_formulation_FEM} has a unique solution and the original subsequence $(\rho_{n_{k}})$ was arbitrary it follows that the entire sequence $(\hat{\rho}(\rho_{n}))$ converges pointwise a.e.\ to $\hat{\rho}(\rho)$. 
\end{proof}

\section{Verification of the abstract theorem assumptions for the self-weight-based formulation }
\label{sec:app_a}

This appendix aims to verify the boundedness and continuity assumptions of the abstract Theorem \ref{teo:convergence} when $J_{P}$ is particularized for the self-weight-based formulation described in Section \ref{sec:self_weight_based_formulation}. We consider the special case treated numerically in Section \ref{sec:num_results} where the design domain is a rectangular box (in 2D or 3D) and the build-direction $\bm{b} = \bm{e}_{d}$. Furthermore, Theorem \ref{thm:stateobjcont} below is proved for a strictly positive lower bound $t_{0}>0$ in the time integration rather than $0$. In numerical computations one must of course use a strictly positive lower bound anyway, so this limitation has no practical consequence. We believe it is possible, but leave it for the moment, to treat the case $t_{0}=0$ theoretically as well, possibly borrowing ideas from the theory of shells where one considers domains shrinking to zero in one direction. Looking at \eqref{eq:boundedness} below it seems that what is needed is an appropriate dependence on the domain height of the Korn's constant and the constant in the extension theorem.

By using the nested form (recall Remark \ref{remark:nested_form}) the cost of the partially built structures becomes
\begin{equation}\label{eq:objective}
    J_P \left( {\left. \rho \right|}_{\Omega_t}, \bm{u} \left({\left. \rho \right|}_{\Omega_t}, t \right), t \right) = 
    L_t \left( \rho; \bm{u} \left( \rho \right) \right)
    \quad
    \textmd{where}
    \quad
    L_{t} \left( \rho; \bm{v} \right) 
    = \int_{\Omega_t} 
    \bm{f}\left( {\left. \rho \right|}_{\Omega_t} \right) \cdot \bm{v} \ud x
    \quad \textmd{for} \quad t \in (0, T].
\end{equation}

The boundedness-assumption on $J_{P}$ is verified using the Cauchy-Schwartz inequality and the fact that $\bm{f}(\rho) \in [L^{2}(\Omega)]^d$:
\begin{equation}\label{eq:Lboundedness}
    \left| L_t \left( \rho; \bm{u} \left( \rho \right) \right) \right| \leq 
    {\left\| \bm{f}\left( {\left. \rho \right|}_{\Omega_t} \right) \right\|}_{0,\Omega_{t}}\,
    {\left\| \bm{u} \left( {\left. \rho \right|}_{\Omega_t} \right) \right\|}_{0,\Omega_{t}} \leq
    {\left\| \bm{f} \left( \rho \right) \right\|}_{0,\Omega}\,
    {\left\| \bm{u} \left( {\left. \rho \right|}_{\Omega_t} \right) \right\|}_{1,\Omega_{t}} =
    c{\left\| \bm{u} \left( {\left. \rho \right|}_{\Omega_t} \right) \right\|}_{1,\Omega_{t}}.
\end{equation}
Then Theorem \ref{thm:stateobjcont} below asserts the existence of weakly convergent subsequences of partially built states and the continuity of $J_{P}$. To prove the theorem we start with two basic results for fixed domains and then proceed to handle the varying domains borrowing techniques from the shape optimization literature \cite{haslinger1996finite}. 

\begin{lemma}\label{lem:Lcontfixeddomain} Let $\bm{f}(\bar{\rho}\left( \rho \right)) = -g_{p} \bar{\rho} \bm{b}$, with the constant $g_{p} < \infty$. Then for a fixed $t \in (0,T]$ and a sequence ($\rho_{n},\bm{v}_{n}$) in $\mathcal{D} \times V_t$ converging weakly$^*$ $\times$ weakly to $(\rho,\bm{v}) \in \mathcal{D} \times V_t$, $L_t \left( \rho_n; \bm{v}_n \right)$ tends to $L_t \left( \rho; \bm{v} \right)$.
\end{lemma}
\begin{proof} Adding and subtracting terms and then using Cauchy-Schwartz gives
\begin{multline}\label{eq:Lcontfixed}
\left| L_t \left( \rho_n; \bm{v}_n \right)-L_t \left( \rho; \bm{v} \right) \right| = \left| \int_{\Omega_{t}} -g_{p} \bar{\rho}(\rho) \bm{b} \cdot \bm{v} \ud x-L_t \left( \rho; \bm{v} \right) \right| = \\
\left| \int_{\Omega_{t}} -g_{p} [\bar{\rho}(\rho_n)-\bar{\rho}(\rho)]  \bm{b} \cdot \bm{v}_{n} \ud x +  \int_{\Omega_{t}} -g_{p} \bar{\rho}(\rho) \bm{b} \cdot \bm{v}_{n} \ud x -L_t \left( \rho; \bm{v} \right) \right| \leq \\
||\bm{b}\cdot \bm{v}_{n}||_{0,\Omega_{t}}\left(\int_{\Omega_{t}} (g_{p} [\bar{\rho}(\rho_n)-\bar{\rho}(\rho)])^2 \ud x\right)^{1/2} + 
\left| L_t \left( \rho; \bm{v}_{n} \right) -L_t \left( \rho; \bm{v} \right) \right|.
\end{multline}
Since $g_{p}^2 \geq (g_{p} [\bar{\rho}(\rho_n)-\bar{\rho}(\rho)])^2 \rightarrow 0$ a.e.,
the first term in the last line tends to zero by the boundeness of the sequence $(\bm{v}_{n})$ and Lebesgue's dominated convergence theorem. The second term in the last line of \eqref{eq:Lcontfixed} tends to zero due to the weak convergence of $(\bm{v}_{n})$ to $\bm{v}$.  
\end{proof}

Using similar techniques as in the proof of Lemma \ref{lem:Lcontfixeddomain} one obtains the following continuity result for the bi-linear form in the sub-problems on a fixed domain (c.f. \cite{petersson1998slope}) :

\begin{lemma}\label{lem:acontfixeddomain} Let $a_{\Omega_{t}}(\rho; \bm{u},\bm{v})$ denote the bi-linear form defined in \eqref{eq:bilinear_linear_form_subproblems_self_weight}. Then for a fixed $t \in (0,T]$ and a sequence ($\rho_{n},\bm{u}_{n}$) in $\mathcal{D}$, $a_{\Omega_{t}}(\rho_{n}; \bm{u}_{n},\bm{v}) \rightarrow a_{\Omega_{t}}(\rho; \bm{u},\bm{v})$ for all $\bm{v} \in V_{t}$.
\end{lemma}

We are now ready to prove

\begin{theorem}\label{thm:stateobjcont} Let $t \in [t_{0},T]$, $t_{0}>0$, be fixed and consider a sequence $(t_{i},\rho_{i})$ in $[t_{0},T] \times \mathcal{D}$ converging strongly $\times$ weakly$^*$ to some $(t,\rho) \in (0,T] \times \mathcal{D}$. Then there is a subsequence such that the corresponding states $\bm{u}_{i} = \bm{u}_{i}(t_{i},\rho_{i})$ converges to a limit  $\bm{u} \in V_{t}$ satisfying $a_{t} \left( \rho; \bm{u}, \bm{v} \right) = L_{t} \left( \rho; \bm{v} \right)$ for all $\bm{v} \in V_{t}$ and such that $L_{t_i}(\rho_{i},\bm{u}_{i}) \rightarrow L_{t}(\rho,\bm{u})$.
\end{theorem}

\begin{proof} The states $\bm{u}_{i} \in V_{i}$ satisfies
\begin{equation}
    a_{i} \left( \rho_{i}; \bm{u}_{i}, \bm{v}_{i} \right) = 
    L_{i} \left(\rho_{i}; \bm{v}_{i} \right), \quad \forall \bm{v}_{i} \in V_{i}.
\label{eq:discrstate1}
\end{equation}
Being bounded with Lipschitz boundary, the domains $\Omega_{i} \equiv \Omega_{t_{i}}$ possess the uniform extension property. We therefore have a Korn's inequality \cite{hlavavcek1989inequalities} with constant depending not on the particular domains, but only on properties shared by all of them. It follows that $a_{i} \left( \rho_{i}; \cdot, \cdot \right)$ is coercive over $V_{i}$ with constant not depending on $\Omega_{i}$. Using \eqref{eq:Lboundedness} then gives
\begin{equation}
    \gamma {\left\| \bm{u}_{i} \right\|}_{1,\Omega_{t_{i}}}^2 \leq
    a_{\Omega_{t_{i}}}(\rho_{i}; \bm{u}_{i}, \bm{u}_{i} ) \leq
    c {\left\| \bm{u}_{i} \right\|}_{1,\Omega_{t_{i}}} \Rightarrow 
    {\left\| \bm{u}_{i} \right\|}_{1,\Omega_{t_{i}}} \leq c/\gamma.
\label{eq:boundedness}
\end{equation} 
The uniform extension property implies that $\bm{u}_{i}$ can be extended to a function $\tilde{\bm{u}}_{i} \in [H^{1}(\Omega)]^d$ such that ${\left\| \tilde{\bm{u}}_{i} \right\|}_{1,\Omega} \leq c_{2} {\left\| \bm{u}_{i} \right\|}_{1,\Omega_{i}}$ with constant $c_{2}$ independent of $i$, whence it follows that ${\left\| \tilde{\bm{u}}_{i} \right\|}_{1, \Omega} \leq c_{2} c/\gamma$. The sequence $\tilde{\bm{u}}_{i}$ will thus have a weakly convergent subsequence converging to some $\tilde{\bm{u}} \in [H^{1}(\Omega)]^d$. Since $\bm{u}_{i} = \bm{0}$ on $\Gamma_{0}$ and the trace operator is weakly continuous it follows that $\bm{u}$, the restriction of $\tilde{\bm{u}}$ to $\Omega_{t}$, satisfies $\bm{u} = \bm{0}$ on $\Gamma_{0}$ and thus that  $\bm{u} \in V_{t}$.

Before proceeding we note that the test-space in \eqref{eq:discrstate1} can be taken as the fixed space $V_{0} = \{ \bm{v} \in [H^{1}(\Omega)]^{d} \;|\; \bm{v} = \bm{0} \in \Gamma_{0} \}$; clearly $\bm{v} \in V_{0}$ implies $\bm{v}|_{\Omega_{i}} \in V_{i}$ and, due to the uniform extension property, every $\bm{v} \in V_{i}$ can be extended to a function in $V_{0}$. Introducing $V_{0}$ is not necessary but simplifies the arguments below.

To handle the fact that the domain varies as $i$ tends to infinity in \eqref{eq:discrstate1} we split $\Omega_{i}$ into three parts (see Fig.~\ref{fig:setdivision}) according to
\begin{equation}\nonumber
\Omega_{i} = G_{m} \cup (\Omega_{i} \setminus \Omega_{t}) \cup ((\Omega_{t} \setminus G_{m}) \cap \Omega_{i}).
\end{equation}
With integer $m\geq \lceil T/Ht \rceil$, $G_{m} = \{ \bm{x} \in \Omega \,|\, 0 < x_{d} < Ht/T - 1/m \}$ will be a subset of $\Omega_{t}$, and for all sufficiently large $i$:s it will thus also be a subset of $\Omega_{i}$ (specifically, $\exists N$ such that $i>N \Rightarrow |t-t_{i}| < T/mH$). 
\begin{figure}[H]
\centering
\includegraphics[scale=0.85]{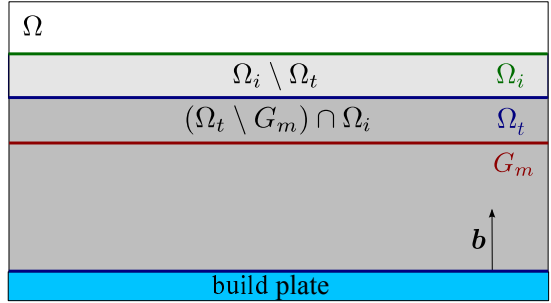}
\caption{Division of $\Omega_{i}$ into a ``big'', fixed part $G_{m}$ and two ``thin'' slices depending on $i$.}
\label{fig:setdivision}
\end{figure}

Now letting $\bm{v} \in V_{0}$ be arbitrary and noting that for a given domain $\omega \subset \Omega$ and design $\rho \in \mathcal{D}$, there is a constant $c$ such that $|a_{\omega}(\rho; \bm{u},\bm{v})| \leq c||\bm{u}||_{1,\omega}||\bm{v}||_{1,\omega}$,
we find that
\begin{multline}\nonumber
|a_{\Omega_{t}}(\rho;\bm{u},\bm{v})-a_{\Omega_{t_{i}}}(\rho_{i};\bm{u}_{i},\bm{v})| \leq 
|a_{\Omega_{t}}(\rho;\bm{u},\bm{v})-a_{\Omega_{G_{m}}}(\rho_{i};\bm{u}_{i},\bm{v})| + |a_{\Omega_{i} \setminus \Omega_{t}}(\rho_{i};\bm{u}_{i},\bm{v})| + |a_{(\Omega_{t} \setminus G_{m}) \cap \Omega_{i}}(\rho_{i};\bm{u}_{i},\bm{v})| \leq \\
|a_{\Omega_{t}}(\rho;\bm{u},\bm{v})-a_{\Omega_{G_{m}}}(\rho_{i};\bm{u}_{i},\bm{v})| + c||\bm{u}_{i}||_{1,\Omega_{i} \setminus \Omega_{t}}||\bm{v}||_{1,\Omega_{i} \setminus \Omega_{t}} + 
c||\bm{u}_{i}||_{1,(\Omega_{t} \setminus G_{m}) \cap \Omega_{i}}||\bm{v}||_{1,(\Omega_{t} \setminus G_{m}) \cap \Omega_{i}} \leq \\
|a_{\Omega_{t}}(\rho;\bm{u},\bm{v})-a_{\Omega_{G_{m}}}(\rho_{i};\bm{u}_{i},\bm{v})| + c||\bm{u}_{i}||_{1,\Omega_{i}}||\bm{v}||_{1,\Omega_{i} \setminus \Omega_{t}} + 
c||\bm{u}_{i}||_{1,\Omega_{i}}||\bm{v}||_{1,\Omega_{t} \setminus G_{m}}.
\end{multline}
Using \eqref{eq:boundedness} on the last two terms then gives
\begin{multline}\nonumber
|a_{\Omega_{t}}(\rho;\bm{u},\bm{v})-a_{\Omega_{t_{i}}}(\rho_{i};\bm{u}_{i},\bm{v})| \leq |a_{\Omega_{t}}(\rho;\bm{u},\bm{v})| + c_{1}||\bm{v}||_{1,\Omega_{i} \setminus \Omega_{t}} + 
c_{2}||\bm{v}||_{1,\Omega_{t} \setminus G_{m}} \leq \\
|a_{\Omega_{t}}(\rho;\bm{u},\bm{v})-a_{\Omega_{G_{m}}}(\rho;\bm{u},\bm{v})| + |a_{\Omega_{G_{m}}}(\rho;\bm{u},\bm{v}) - a_{\Omega_{G_{m}}}(\rho_{i};\bm{u}_{i},\bm{v})| + c_{1}||\bm{v}||_{1,\Omega_{i} \setminus \Omega_{t}} + 
c_{2}||\bm{v}||_{1,\Omega_{t} \setminus G_{m}},
\end{multline}
where the first term in the second line converges to zero since $\Omega_{G_{m}}$ converges to $\Omega_{t}$ and $\rho$ and $\bm{u}$ {\color{blue}are fixed; the second tends to zero using} Lemma \ref{lem:acontfixeddomain}; and the last terms because $|\Omega_{i} \setminus \Omega_{t}| \rightarrow 0$ and $|\Omega_{t} \setminus G_{m}| \rightarrow 0$.

The right-hand side in \eqref{eq:discrstate1} appears also in the objective [recall \eqref{eq:objective}], so we need to consider a case where also the second argument varies. More precisely, for the same sequence of $\rho_{i}$:s as before and a sequence $\bm{v}_{i} \in V_{i}$ converging weakly to some $\bm{v} \in V_{t}$, 
\begin{multline}\nonumber
|L_{\Omega_{t}}(\rho;\bm{v})-L_{\Omega_{t_{i}}}(\rho_{i};\bm{v}_{i})| = 
|L_{\Omega_{t}}(\rho;\bm{v})-L_{G_{m}}(\rho_{i};\bm{v}_{i})| +
|L_{\Omega_{i} \setminus \Omega_{t}}(\rho_{i};\bm{v}_{i})| +
|L_{(\Omega_{t} \setminus G_{m}) \cap \Omega_{i}}(\rho_{i};\bm{v}_{i})|  \leq \\
|L_{\Omega_{t}}(\rho;\bm{v})-L_{G_{m}}(\rho_{i};\bm{v}_{i})| + ||\bm{f}||_{0,\Omega_{i} \setminus \Omega_{t}}||\bm{v}_{i}||_{0,\Omega_{i} \setminus \Omega_{t}} + ||\bm{f}||_{0,\Omega_{t} \setminus G_{m}}||\bm{v}_{i}||_{0,(\Omega_{t} \setminus G_{m}) \cap \Omega_{i}}  \leq \\
|L_{\Omega_{t}}(\rho;\bm{v})-L_{G_{m}}(\rho;\bm{v})| + |L_{G_{m}}(\rho;\bm{v}) - L_{G_{m}}(\rho_{i};\bm{v}_{i})| + ||\bm{f}||_{0,\Omega_{i} \setminus \Omega_{t}}||\bm{v}_{i}||_{0,\Omega_{i} \setminus \Omega_{t}} + ||\bm{f}||_{0,\Omega_{t} \setminus G_{m}}||\bm{v}_{i}||_{0,(\Omega_{t} \setminus G_{m}) \cap \Omega_{i}},
\end{multline}
where the first term in the last line tends to zero since $\Omega_{G_{m}}$ converges to $\Omega_{t}$ and $\rho$ and $\bm{v}$ are fixed; the second to zero using Lemma \ref{lem:Lcontfixeddomain}; and the last terms tends to zero since $\bm{f}$ is fixed and $(\bm{v}_{i})$ is a convergent, hence bounded sequence.

Having shown that the left and right-hand sides in \eqref{eq:discrstate1} converge, for every $\bm{v} \in V_{t}$, to $a_{t}(\rho;\bm{u},\bm{v})$  and $L_{t}(\rho;\bm{v})$ respectively, it follows that the limiting state $\bm{u} \in V_{t}$ satisfies $a_{t} \left( \rho; \bm{u}, \bm{v} \right) = L_{t} \left( \rho; \bm{v} \right)$ for all $\bm{v} \in V_{t}$. 
\end{proof}

\end{document}